\documentclass{amsart}

\usepackage{amsmath}
\usepackage{amssymb}
\usepackage{amsfonts}
\usepackage{amsthm}
\usepackage{cancel}
\usepackage{amsrefs}
\usepackage{paralist}  % Import compactenum class
\usepackage{booktabs}  % For table
\usepackage{graphicx}  % including eps graphics

\newcommand{\R}{\mathbb{R}}
\newcommand{\inv}{^{-1}}
\newcommand{\ol}{\overline}

\newcommand{\sm}{\setminus}
\DeclareMathOperator{\Div}{\text{div}}
\renewcommand{\tilde}{\widetilde}

\newtheorem{thm}{Theorem}
\newtheorem{definition}[thm]{Definition}
\newtheorem{lemma}[thm]{Lemma}
\newtheorem{prop}[thm]{Proposition}
\newtheorem{cor}[thm]{Corollary}
\newtheorem{conj}[thm]{Conjecture}

\begin{document}
\title[Zero area singularities]{A geometric theory of zero area singularities\\ in general relativity}
\date{January 25, 2012}
\author{Hubert L. Bray}
\address{Mathematics Department\\
Duke University, Box 90320\\
Durham, NC 27708-0320\\}
\email{bray@math.duke.edu}
\thanks{The research of the first author was supported in part by NSF grant \#DMS-0706794.}
\author{Jeffrey L. Jauregui}
\address{Department of Mathematics\\
David Rittenhouse Lab\\
209 South 33rd Street\\
Philadelphia, PA 19104\\}
\email{jauregui@math.upenn.edu}

\begin{abstract}
The Schwarzschild spacetime metric of negative mass is well-known to contain a naked singularity.  In a spacelike 
slice, this singularity of the metric is characterized by the property that nearby surfaces 
have arbitrarily small area.
We develop a theory of such ``zero area singularities'' in Riemannian manifolds, generalizing far beyond the
Schwarzschild case (for example, allowing the singularities to have nontrivial topology).
We also define the mass of such singularities. The main result of this paper is a lower bound
on the ADM mass of an asymptotically flat manifold of nonnegative scalar curvature in terms of the
masses of its singularities, assuming a certain conjecture in conformal geometry.
The proof relies on the Riemannian Penrose inequality \cite{bray_RPI}.  Equality is attained in the 
inequality by the Schwarzschild metric of negative mass.  An immediate corollary is a version of the 
positive mass theorem that allows for certain types of incomplete metrics.
\end{abstract}

\maketitle
\section{Introduction: the negative mass Schwarzschild metric}
The first metrics one typically encounters in the study of general relativity are the 
Minkowski spacetime metric and the Schwarzschild spacetime metric,
the latter given by
\begin{equation*}
ds^2 = -\left(1 - \frac{2m}{R}\right) dt^2 + \left(1 - \frac{2m}{R}\right)\inv dR^2 + R^2 
(d\theta^2 + \sin^2 \theta \;d\phi^2),
\end{equation*}
where $t$ is the time coordinate, $(R, \theta,\phi)$
are spatial spherical coordinates, and $m$ is some positive number. This
represents the exterior region $R > 2m$ of a non-rotating black hole of mass $m$ in vacuum.  A 
spacelike slice of this Lorentzian metric can be obtained by taking a level set of 
$t$; under a coordinate transformation $R=r\left(1+\frac{m}{2r}\right)^2$, the resulting 
3-manifold is isometric to $\R^3$ minus the ball of radius $m/2$ about the 
origin, with the conformally flat metric
\begin{equation}
ds^2 = \left(1 + \frac{m}{2r}\right)^4 \delta,
\label{riem_schwarz_metric}
\end{equation}
where $\delta=dx^2+dy^2+dz^2$ is the usual flat metric on $\R^3$ and $r=\sqrt{x^2+y^2+z^2}$.
Its boundary is a minimal surface that represents the apparent horizon of the black 
hole.  We refer to (\ref{riem_schwarz_metric}) as the \emph{Schwarzschild metric} (of mass $m$).

Consider instead the metric $(\ref{riem_schwarz_metric})$ with $m<0$.  This gives a Riemannian 
metric on $\R^3$ minus a closed ball of 
radius $|m|/2$ about the origin that approaches zero near its inner boundary.  One may loosely
think of this manifold as a slice of a spacetime with a single ``black hole of 
negative mass.''  In fact, this metric has a naked singularity, as the singularity on the 
inner boundary is not enclosed
by any apparent horizon.  In this paper we introduce a theory of such ``zero area 
singularities'' (ZAS), modeled on the ``Schwarzschild ZAS metric'' (i.e., 
$(\ref{riem_schwarz_metric})$ 
with $m<0$), yet far more 
general.  Some of the problems we address are: 
\begin{compactenum}
\item When can such singularities be ``resolved''?
\item What is a good definition of the mass of such a singularity?
\item Can the ADM mass of an asymptotically flat manifold of nonnegative scalar curvature be
estimated in terms of the masses of its singularities?
\end{compactenum}
The third question is motivated by the positive mass theorem \cites{schoen_yau,witten} 
and Riemannian Penrose inequality \cites{imcf,bray_RPI} (Theorems \ref{pmt} and \ref{rpi} below).  The former states that, under 
suitable conditions, the ADM mass of an asymptotically flat 
3-manifold is nonnegative, 
with zero mass occurring only for the flat metric on $\R^3$.  The latter
improves this to provide a lower 
bound on the ADM mass  
in terms of the masses of its ``black holes.''  Here, the case of equality is attained by the 
Schwarzschild metric with $m > 0$.  See appendix A for details on asymptotic flatness and ADM 
mass.
 
The main theorem of this paper is the Riemannian ZAS inequality (see Theorems \ref{thm_zas_ineq} and \ref{thm_zas_ineq_full}),
introduced by the first author \cite{bray_npms}.  It is an analog of the Riemannian Penrose inequality, 
but for zero area singularities instead of black holes.
Specifically, this inequality gives a lower bound for the ADM mass of an asymptotically flat manifold in terms of the
masses of its ZAS. Its proof assumes an unproven
conjecture (Conjecture \ref{conj_conformal}) regarding the outermost minimal area 
enclosure of the boundary, with respect to a conformal metric.  While the conjecture is known to be true in some cases,
proving it remains an open problem.  Although we shall write ``the'' Riemannian ZAS inequality in this paper, we remark that other similar
inequalities may be discovered in the future that also deserve this title.  Table \ref{table_thms} illustrates how this theorem fits together with 
the positive mass theorem and Riemannian Penrose inequality.
 
\begin{table}[ht]
\caption{The Schwarzschild metric as a case of equality}
{\small
\begin{tabular}{@{}lll@{}} \toprule
sign($m$) & metric $\left(1+\frac{m}{2r}\right)^4 \delta$ & unique case of equality of\\  
\midrule
$0$ & Euclidean & positive mass theorem\\ 
$+$ & Schwarzschild metric & Riemannian penrose inequality\\ 
$-$ & Schwarzschild ZAS metric & Riemannian ZAS inequality (Theorem \ref{thm_zas_ineq})\\ \bottomrule
\end{tabular}}
\label{table_thms}
\end{table}
For reference, we recall the following theorems.  The geometric assumption of nonnegative
scalar curvature is equivalent, physically, to the dominant energy condition (for 
totally geodesic slices of spacetimes).

\begin{thm}[Positive mass theorem \cite{schoen_yau}] Let $(M,g)$ be a complete asymptotically flat 
Riemannian 3-manifold (without boundary) of nonnegative scalar curvature with ADM mass $m$. 
Then $m \geq 0$, with equality holding if and only if $(M,g)$ is isometric to $\R^3$ with the flat metric.
\label{pmt}
\end{thm}
Witten gave an alternative proof of Theorem $\ref{pmt}$ for spin manifolds \cite{witten}.

\begin{thm}[Riemannian Penrose inequality, Theorem 19 of \cite{bray_RPI}] 
Let $(M,g)$ be a complete asymptotically flat Riemannian 3-manifold with compact smooth 
boundary $\partial M$ and nonnegative scalar curvature, with ADM mass $m$.  
Assume that $\partial M$ is minimal (i.e. has zero mean curvature), and let $S$ be the
outermost minimal area enclosure of $\partial M$.
Then $m \geq \sqrt{\frac{A}{16\pi}}$, where $A$ is the area of $S$.
Equality holds if and only if $(M,g)$ is isometric to the 
Schwarzschild metric of mass $m$ outside of $S$.
\label{rpi}
\end{thm}
 
See appendix A for details on the outermost minimal area enclosure.
Theorem $\ref{rpi}$ was first proved by Huisken and Ilmanen \cite{imcf} with $A$ replaced by 
the area of the largest connected component of $S$.

\subsection{Negative mass in the literature}
The concept of negative mass in both classical physics and general relativity has appeared frequently 
in the literature.  The following list is
a small sample of such articles and is by no means comprehensive.
\begin{itemize}
\item Bondi discusses negative mass in Newtonian mechanics, distinguishing inertial mass, passive gravitational mass, 
and active gravitational mass \cite{bondi}.  He then proceeds to study a two-body problem in general relativity involving 
bodies with masses of opposite sign.
\item Bonnor considers Newtonian mechanics and general relativity under the assumption that all mass is negative \cite{bonnor}.
Included in the discussion are 1) the motion of test particles for a Schwarzschild spacetime of negative mass, 2)
Friedmann-Robertson-Walker cosmology with negative mass density, and 3) charged particles of negative mass.
\item More recently, research has turned toward the question of stability of 
the negative mass Schwarzschild spacetime. Gibbons, Hartnoll, and Ishibashi studied linear gravitational perturbations 
to this metric and found it to be stable for a certain boundary condition on the perturbations \cite{gibbons}.
However, a separate analysis by Gleiser and Dotti reached a different conclusion, indicating the negative mass Schwarzschild 
spacetime to be perturbatively unstable for all boundary conditions \cite{gleiser}.  The papers are mathematically
consistent with each other, with differences arising from subtleties pertaining to defining time evolution in a spacetime 
with a naked singularity.  The issue of stability warrants further analysis, although we do not consider it here.
\end{itemize}

The present paper offers a new perspective on singularities arising from negative mass, extending 
past the Schwarzschild case.
We will restrict our attention to the case of time-symmetric (i.e., totally geodesic), spacelike 
slices of spacetimes.  This setting is a natural starting point, as it was for the positive mass 
theorem and Penrose inequality.

\subsection{Overview of contents}
Before providing an overview of the paper, we emphasize that the statements and proofs of the main theorems
appear near the end, in section \ref{sec_main}.

In section \ref{sec_def} we make precise the notion of zero area singularity.  By
necessity, this is preceded by a discussion of convergence for sequences of surfaces.  Next, we define
two well-behaved classes of ZAS: those that are ``regular'' and ``harmonically regular.''

Section \ref{sec_mass} introduces the mass of a ZAS: this is a numerical quantity that ultimately gives
a lower bound on the ADM mass in the main theorem.  Defining the mass for regular ZAS is straightforward; 
for arbitrary ZAS, formulating a definition requires more care.  We discuss connections between the ZAS mass
and the Hawking mass.  Next, we define the capacity of a ZAS based on the classical notion of harmonic 
capacity.  The important connection between mass and capacity is that if the capacity is positive, then 
the mass is $-\infty$.

Spherically symmetric metrics with zero area singularities are studied in section \ref{sec_examples}.  In this simple setting,
we explicitly compute the mass and capacity.  An example is given that shows the concepts of regular and
harmonically regular ZAS are distinct.  Experts may prefer to skip this section, which is largely computational
and detail-oriented.

The main two theorems, comprising two versions of the Riemannian ZAS inequality, are stated and proved (up to an unproven conjecture)
in section \ref{sec_main}.  An immediate corollary is a version of the positive mass theorem for manifolds with certain
types of singularities.  

After providing one final example, we conclude with a discussion about several related open problems
and conjectures.  Two appendices follow which are referred to as needed.

\subsection{Comments and acknowledgements}
In 1997, the first author, just out of graduate school, sat next to Barry
Mazur at a conference dinner at Harvard, who, quite characteristically, asked
the first author a series of probing questions about his research, which at
the time concerned black holes.  One of the questions was ``Can a black hole
have negative mass, and if so, what properties would it have?'' Contemplating
this natural question marked the beginning of an enjoyable journey leading to
this paper.

The first author initiated this work, originally presented at a conference in 2005 under
the heading ``Negative Point Mass Singularities'' \cite{bray_npms}.  The second author commenced
work on this project as a graduate student, and wrote his thesis on a closely related topic 
\cite{jauregui}.  He would like to thank Mark Stern and Jeffrey Streets for 
helpful discussions.

\section{Definitions and preliminaries}
\label{sec_def}
In this paper, the singularities in question will arise as metric singularities on a boundary 
component of a manifold.  
To study the behavior of the metric near a singularity, we make extensive use of the idea of nearby
surfaces converging to a boundary component.

Throughout this paper $(M,g)$ will be a smooth, asymptotically 
flat Riemannian 3-manifold, with compact, smooth, nonempty boundary $\partial M$ (see appendix A 
for details on asymptotic flatness).  We do not assume that $g$ extends smoothly to $\partial M$.
We make no other restrictions on the topology of $\partial M$ (e.g., 
connectedness, orientability, genus).  

\subsection{Convergence of surfaces}
\label{sec_surfaces}
For our purposes, a \emph{surface} $S$ in $M$ will always mean a $C^\infty$, closed, embedded 2-manifold 
in the interior of $M$ that is the boundary of a bounded open region $\Omega$.  (Note that $\Omega$ is 
uniquely determined by $S$.)  We say that a surface $S_1 = \partial \Omega_1$ 
\emph{encloses} a surface $S_2= \partial \Omega_2$ if $\Omega_1 \supset \Omega_2$.  
If $\Sigma$ is a nonempty subcollection of the components of $\partial M$, we say that a surface $S = 
\partial \Omega$ \emph{encloses} $\Sigma$ if $S$ is homologous to $\Sigma$.

We next define what it means for a surface $S$ to be ``close to'' $\Sigma$ (with $\Sigma$ as above).
Let $U \subset M$ be a neighborhood
of $\Sigma$ that is diffeomorphic to $\Sigma \times [0, a)$ for some $a > 0$.  This gives a  
coordinate system $(x,s)$ on $U$ where $x \in \Sigma$ and $s \in [0,a)$.  If $S \subset U$ is a surface
that can be parameterized in these coordinates as $s=s(x)$, then we say it is a ``graph over $\Sigma$''; clearly
such $S$ encloses $\Sigma$. 
Whenever we discuss the convergence of surfaces, it will be implicit that the surfaces are graphs 
over $\Sigma$.

\begin{definition}
Let $\{S_n\}$ be a sequence of surfaces that are graphs over $\Sigma$ that can be parameterized as $s_n=s_n(x)$ (see above).
We say that $\{S_n\}$ \textbf{converges} in $C^k$ to $\Sigma$ if the functions $s_n:\Sigma\to [0,a)$ 
converge to 0 in $C^k$.
\end{definition}
We emphasize that convergence in $C^k$ depends only on the underlying smooth structure of $M$ and not on the metric.
As an example, $S_n \to \Sigma$ in $C^0$ if for any open set $U$ containing $\Sigma$, there exists $n_0>0$ such
that $S_n \subset U$ for all $n \geq n_0$.  We shall not deal with convergence stronger than $C^2$ and will 
explain the significance of convergence in $C^1$ and $C^2$ as necessary.

\subsection{Zero area singularities}
We now give the definition of zero area singularity.  Both the singular and plural will be 
abbreviated ``ZAS.''
\begin{definition}
Let $g$ be an asymptotically flat metric on $M \sm \partial M$.  A connected component 
$\Sigma^0$ of $\partial M$ is a \textbf{zero area singularity (ZAS)} of $g$ if for every
sequence of surfaces $\{S_n\}$ converging in $C^1$ to $\Sigma^0$, the areas of $S_n$
measured with respect to $g$ converge to zero.
\label{def_zas}
\end{definition}
Topologically, a ZAS is a boundary surface in $M$, not a point.  However, in terms
of the metric, it is often convenient to think of a ZAS as a point formed by shrinking the 
metric to zero.  For example, the boundary sphere of the Schwarzschild ZAS metric is a ZAS.  
Also, most notions of ``point singularity'' are ZAS (after deleting the point). A depiction of 
a manifold with ZAS is given in figure \ref{fig_ex_zas}.

ZAS could be defined for manifolds that are not asymptotically flat, but we do not pursue this direction.

\begin{figure}[ht]
\caption{A manifold with zero area singularities}
\begin{center}
\includegraphics[scale=0.5]{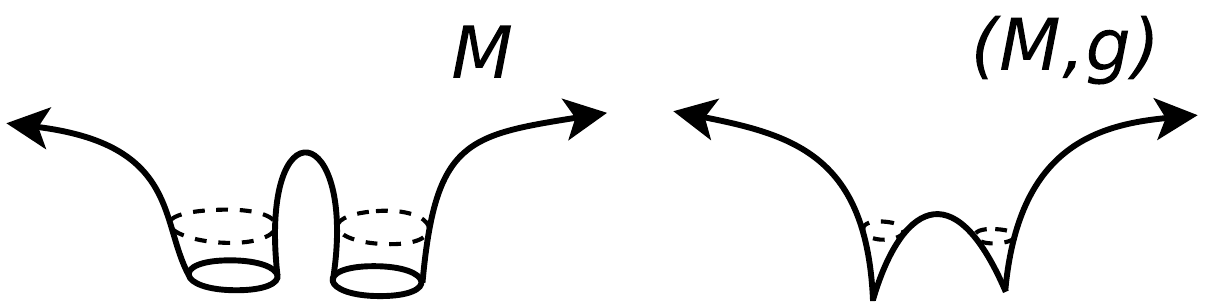}
\end{center}
\flushleft\footnotesize{On the left is a drawing of an abstract smooth manifold $M$ with two boundary components. On the right is
a drawing of the same manifold equipped with a metric $g$ for which both boundary components are ZAS.  The dotted lines represent
cross-sectional surfaces (not necessarily 2-spheres).}
\label{fig_ex_zas}
\end{figure}

In the case that $g$ extends continuously to the boundary, we have several equivalent conditions for ZAS:
\begin{prop}
Suppose $\Sigma^0$ is a component of $\partial M$ to which $g$ extends continuously as a 
symmetric 2-tensor.  The following are equivalent:
\begin{compactenum}
\item $\Sigma^0$ is a ZAS of $g$.
\item $\Sigma^0$ has zero area measured with respect to $g$ (see below).
\item For each point $x \in \Sigma^0$, $g$ has a null eigenvector tangent to 
$\Sigma^0$ at $x$.
\item There exists a sequence of surfaces $\{\Sigma_i\}$ converging in $C^1$ to $\Sigma^0$ 
such that $|\Sigma_i|_g$ converges to zero.
\end{compactenum}
Here, $|\Sigma_i|_{g}$ is the area of $\Sigma_i$ measured with respect to $g$.
\label{prop_zas_criteria}
\end{prop}
A continuous, symmetric 2-tensor $k$ on 
a surface that is positive semi-definite can 
be used to compute areas by integrating 
the 2-form defined locally in coordinates by $\sqrt{\det k_{ij}} \,dx_1 \wedge dx_2$.  

\begin{proof} The proof is an immediate consequence of the following observations:
\begin{itemize}
 \item If $g$ extends continuously to $\Sigma^0$, then for any sequence of surfaces 
$\{\Sigma_n\}$ converging in $C^1$ to $\Sigma^0$, the areas converge: $|\Sigma_n|_g \to |\Sigma^0|_g$.
\item If $k$ is the restriction of $g$ to the tangent bundle of $\Sigma^0$, 
then $\det k_{ij}=0$ at $p$ if and only if $k$ has an eigenvector with zero eigenvalue at $p$.
\end{itemize}
\end{proof}

In general, it is not necessary that $g$ extend continuously to the boundary in the definition of ZAS.

\subsection{Resolutions of regular singularities}
We now discuss what it means to ``resolve'' a zero area singularity.  An important case of ZAS
occurs when a smooth metric on $M$ is deformed by a conformal factor that vanishes on the boundary \cite{bray_npms}:
\begin{definition}
Let $\Sigma^0$ be a ZAS of $g$. Then $\Sigma^0$ is \textbf{regular} if there exists 
a smooth, nonnegative function $\ol \varphi$ and a smooth metric $\ol g$, both defined on a neighborhood 
$U$of $\Sigma^0$, such that
\begin{compactenum}
\item $\ol \varphi$ vanishes precisely on $\Sigma^0$,
\item $\ol \nu(\ol \varphi) > 0$ on $\Sigma^0$, where $\ol \nu$ is the unit normal to $\Sigma^0$ (taken with 
respect to $\ol g$ and pointing into the manifold), and
\item $g = \ol \varphi^4 \ol g$ on $U\sm \Sigma^0$.
\end{compactenum}
If such a pair $(\ol g, \ol \varphi)$ exists, it is called a \textbf{local resolution} of $\Sigma^0$.  
\label{def_local_resolution}
\end{definition}
The significance of the condition $\ol \nu(\ol \varphi) > 0$ is explained further in Lemma \ref{lemma_smooth_extension} and
is crucial in the proof of Proposition \ref{prop_reg_mass_well_def}.  As an example, the Schwarzschild ZAS with $m < 0$ is a regular ZAS 
with a local resolution $(\ol g, \ol \varphi)$, where $\ol g$ is the flat metric and $\ol \varphi = \left(1 + \frac{m}{2r}\right)$.  
A graphical depiction of a local resolution is given in figure \ref{fig_resolution}.
\begin{figure}[ht]
\caption{A resolution of a regular ZAS}
\begin{center}
\includegraphics[scale=0.5]{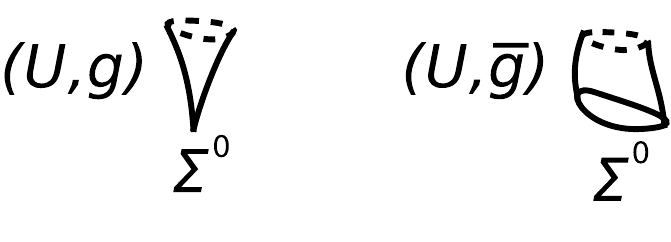}
\end{center}
\flushleft\footnotesize{On the left is a neighborhood $U$ of a 
regular ZAS $\Sigma^0$ in the metric $g$.
On the right is the same neighborhood $U$ equipped with a resolution metric $\ol g$.
The metrics $g$ and $\ol g$ are conformal in $U\sm \Sigma^0$, with $g = \ol \varphi^4 \ol g$, where $\ol 
\varphi$ vanishes on $\Sigma^0$.}
\label{fig_resolution}
\end{figure}

Much of our work utilizes a nicer class of singularities: those for which the resolution function can be chosen
to be harmonic.
\begin{definition}
A regular ZAS $\Sigma^0$ of $g$ is said to be \textbf{harmonically regular} if there exists a local resolution $(\ol g, \ol \varphi)$
such that $\ol \varphi$ is harmonic with respect to $\ol g$.  Such a pair $(\ol g, \ol \varphi)$ is called a \textbf{local harmonic resolution}.
\end{definition}

In the case of a local harmonic resolution, the 
condition $\ol \nu (\ol \varphi) > 0$ holds automatically by the maximum principle.  
We remark that if one local resolution (or local harmonic resolution) exists, then so do infinitely many.

The Schwarzschild ZAS is harmonically regular, since the function $\left(1+\frac{m}{2r}\right)$ is harmonic with respect
to the flat metric on $\R^3$.  In section \ref{sec_examples} we give examples of ZAS that 
are not regular and ZAS that are regular but not harmonically regular.

If several components of $\partial M$ are (harmonically) regular ZAS, then there is a natural
notion of a local (harmonic) resolution of the union $\Sigma$ of these components: in Definition \ref{def_local_resolution},
simply replace $\Sigma^0$ with $\Sigma$.

Since our ultimate goal---the Riemannian ZAS inequality---is a global geometric statement, we require resolutions 
that are globally defined.  
\begin{definition}
Suppose all components of $\Sigma = \partial M$ are harmonically regular ZAS.
Then the pair $(\ol g, \ol \varphi)$ is a \textbf{global harmonic resolution} of $\Sigma$ if
\begin{compactenum}
\item $\ol g$ is a smooth, asymptotically flat metric on $M$,
\item $\ol \varphi$ is the $\ol g$- harmonic function on $M$ vanishing on $\Sigma$ and tending to 
one at infinity, and
\item $g = \ol \varphi^4 \ol g$ on $M \sm \Sigma$.
\end{compactenum}
\label{def_global_harmonic_resolution}
\end{definition}
For example, the aforementioned resolution of the Schwarzschild ZAS is a global harmonic 
resolution.  In general, if $\Sigma$ consists of harmonically regular ZAS, it is not clear that a 
global harmonic resolution exists; however, this is known to be true:
\begin{prop}[Theorem 58 of \cite{jauregui}]
\label{prop_ghr_exists}
If $\Sigma=\partial M$ is a collection of harmonically regular ZAS in $(M,g)$, then $\Sigma$ admits a global
harmonic resolution.
\end{prop}

\section{Mass and capacity of ZAS}
\label{sec_mass}
In a time-symmetric (i.e., totally geodesic) spacelike slice of a spacetime, we adopt the viewpoint that
black holes may be identified with apparent horizons.  An
apparent horizon is defined to be a connected component of the outermost 
minimal surface in the spacelike slice.
If $A$ is the area of an apparent horizon $S$, then its mass (or ``black hole mass'') is defined to 
be $m_{BH}(S)=\sqrt{\frac{A}{16\pi}}$.  This definition has physical 
\cite{penrose} and 
mathematical \cites{imcf,bray_RPI} motivation; it also
equals $m$ for the apparent horizon in the Schwarzschild metric of mass $m > 0$.  We note the 
black hole mass is also given by the limit of the Hawking masses of a sequence of surfaces converging in 
$C^2$ to the apparent horizon.  
Recall the Hawking mass of any surface $S$ in $(M,g)$ is given by
\begin{equation*}
m_H(S) := \sqrt{\frac{|S|_g}{16\pi}}\left(1 - \frac{1}{16\pi} \int_S H^2 dA\right),
\end{equation*}
where $|S|_{g}$ is the area of $S$ with respect to $g$, $H$ is the mean curvature of $S$, 
and $dA$ is the area form on $S$ induced by $g$.  The significance of $C^2$ convergence is explained
in the proof of Proposition \ref{prop_H43} below.

Defining the mass of a ZAS, on the other hand, is not as straightforward, since the metric
becomes degenerate and potentially loses some regularity at the boundary. For regular ZAS, it is 
possible to define mass in terms of a local resolution in such a way as to not 
depend on the choice of local resolution.  In the general case, defining the mass is more 
involved.  We first consider the regular case.

\subsection{The mass of regular ZAS}
\label{sec_reg_mass}
Following \cite{bray_npms}, we define the mass of a regular ZAS:
\begin{definition}
Let $(\ol g, \ol \varphi)$ be a local resolution of a ZAS $\,\Sigma^0$ of $g$.  
Then the \textbf{regular mass} of $\Sigma^0$ is defined by the integral
\begin{equation}
m_{\text{\emph{reg}}}\left(\Sigma^0\right) = -\frac{1}{4}\left(\frac{1}{\pi}\int_{\Sigma^0} 
\ol\nu(\ol 
\varphi)^{4/3} d\ol A\right)^{3/2},
\label{def_reg_mass}
\end{equation}
where $\ol \nu$ is the unit normal to $\Sigma^0$ (pointing into the manifold) and $d\ol A$ is the 
area form induced by $\ol g$.
\end{definition}
The advantages of this definition are that it
\begin{compactenum}
\item is independent of the choice of local resolution (Proposition \ref{prop_reg_mass_well_def}),
\item depends only on the local geometry of $(M,g)$ near $\Sigma^0$ (Proposition \ref{prop_H43}),
\item is related to the Hawking masses of nearby surfaces (Proposition \ref{reg_mass_thm}),
\item arises naturally in the proof of the Riemannian ZAS inequality (Theorem \ref{thm_zas_ineq}), and
\item equals $m$ for the Schwarzschild ZAS metric of ADM mass $m < 0$ (left to the reader).
\end{compactenum}
Before elaborating on these issues, we
take a moment to describe how the masses of regular ZAS add together.  If $\Sigma^1, \ldots, \Sigma^k$ are regular ZAS and
$\Sigma = \cup_{i=1}^k \Sigma^i$, then applying $(\ref{def_reg_mass})$ to $\Sigma$ gives
\begin{equation}
m_{\text{reg}}\left(\Sigma\right) = - \left(\sum_{i=1}^k m_{\text{reg}}\left(\Sigma^i\right)^{2/3}\right)^{3/2}.
\label{eq_add_masses}
\end{equation}
This is analogous with the case of black holes: if $\Sigma^1, \ldots, \Sigma^k$ 
are apparent horizons with black hole masses $m_i = \sqrt{\frac{|\Sigma^i|}{16\pi}}$,
and if the black hole mass of their union $\Sigma$ is defined to be 
$\sqrt{\frac{|\Sigma|}{16\pi}}$ 
(c.f. \cite{penrose}), then
\begin{equation*}
m_{BH}\left(\Sigma\right) = \left(\sum_{i=1}^k m_i^2\right)^{1/2}.
\end{equation*}
Next, we show the regular mass is well-defined. (This was also proved in \cites{bray_npms,robbins}.)  
First, we require the following lemma:
\begin{lemma}
If $(\ol g_1, \ol \varphi_1)$ and $(\ol g_2, \ol \varphi_2)$ are two local resolutions of a regular ZAS $\Sigma^0$,
then the ratio $\frac{\ol \varphi_2}{\ol \varphi_1}$ extends smoothly to $\Sigma^0$ as a strictly positive function.
\label{lemma_smooth_extension}
\end{lemma}
\begin{proof}
Note that $\ol \varphi_1$ and $\ol \varphi_2$ vanish on $\Sigma^0$ and have nonzero normal derivative there.  The
proof follows from considering Taylor series expansions for coordinate expressions of $\ol \varphi_1$ 
and $\ol \varphi_2$ near $\Sigma^0$.
\end{proof}

\begin{prop}
The definition of $m_{\text{reg}}(\Sigma^0)$ is independent of the choice of local resolution.
\label{prop_reg_mass_well_def}
\end{prop}
\begin{proof}
Let $(\ol g_1, \ol \varphi_1)$ and $(\ol g_2, \ol \varphi_2)$ be two local resolutions of $\Sigma^0$, defined on a neighborhood $U$
of $\Sigma^0$.  Then on $U \sm \Sigma^0$,
\begin{equation*}
\ol \varphi_1^4 \, \ol g_1 = g = \ol \varphi_2^4 \,\ol g_2.
\end{equation*}
By Lemma \ref{lemma_smooth_extension}, $\lambda:= \frac{\ol \varphi_1}{\ol \varphi_2}$ is smooth and positive on $U$.
In particular, $\ol g_2 = \lambda^4 \ol g_1$ on $U$.  This allows us to compare area elements $ d\ol A_i$ and unit normals $\ol \nu_i$ on $\Sigma^0$ 
in the metrics $\ol g_1$ and $\ol g_2$:
\begin{align*}
d\ol A_2 &= \lambda^4 d\ol A_1,\\
\ol \nu_2 &= \lambda^{-2}\ol \nu_1.
\end{align*}
We show the integrals in $(\ref{def_reg_mass})$ are the same whether computed for $(\ol g_1, \ol \varphi_1)$ or $(\ol g_2, \ol \varphi_2)$.
\begin{align*}
\int_{\Sigma^0} \ol\nu_2(\ol \varphi_2)^{4/3} d\ol A_2 &= \int_{\Sigma^0} \left(\lambda^{-2} 
		\ol\nu_1\left(\frac{\ol \varphi_1}{\lambda}\right)\right)^{4/3} \lambda^4 d\ol A_1\\
	&= \int_{\Sigma^0} \left(\lambda^{-2}\left( \frac{\ol\nu_1 (\ol \varphi_1)}{\lambda} - 
		\cancel{\ol \varphi_1 \frac{\ol\nu_1(\lambda)}{\lambda^2}}\right)\right)^{4/3} \lambda^4 d\ol A_1\\
	&= \int_{\Sigma^0} \ol \nu_1(\ol \varphi_1)^{4/3} d\ol A_1,
\end{align*}
where the cancellation occurs because $\ol \varphi_1$ vanishes on $\Sigma^0$.
\end{proof}
The next few results serve to: give alternate characterizations
of the regular mass, relate the regular mass to the Hawking mass, and provide motivation for 
the definition of the mass of an arbitrary ZAS.
We have placed an emphasis on 
the Hawking mass (versus other quasi-local mass functionals) due to its relevance to the Riemannian 
Penrose inequality \cite{imcf} and its role in the proof of the Riemannian ZAS inequality
for the case of a single ZAS \cite{robbins}.  

\begin{prop}
Let $\Sigma$ be a subset of $\partial M$ consisting of regular ZAS of $g$.
If $\{\Sigma_n\}$ is a sequence of surfaces converging in $C^2$ to $\Sigma$, then
\begin{equation*}
m_{\text{\emph{reg}}}(\Sigma) = -\lim_{n \to \infty} \left(\frac{1}{16\pi} \int_{\Sigma_n} H^{4/3} dA 
\right)^{3/2}.
\end{equation*}
In particular, the right side is independent of the choice of sequence, and the left side depends 
only on the local geometry of $(M,g)$ near $\Sigma$.
\label{prop_H43}
\end{prop}
\begin{proof}
Let $(\ol g, \ol \varphi)$ be some local resolution of $\Sigma$.  Apply formula 
$(\ref{eq_conf_mean_curv})$ in
appendix B for the change in mean curvature of a hypersurface under a conformal change of the 
ambient metric.  Below, $H$ and $\ol H$ are the mean curvatures of $\Sigma_n$
in the metrics $g$ and $\ol g$, respectively.
\begin{align*}
-\left(\frac{1}{16\pi} \int_{\Sigma_n} H^{4/3} dA \right)^{3/2}
        &= -\left(\frac{1}{16\pi}
           \int_{\Sigma_n} \Big(\ol \varphi^{\,-2} \ol H + 4\ol \varphi^{\,-3} \ol \nu(\ol \varphi)\Big)^{4/3} 
	   \ol \varphi^{\,4} d\ol A \right)^{3/2}\\
        &= -\left(\frac{1}{16\pi}
           \int_{\Sigma_n} \Big(\ol \varphi \ol H + 4\ol \nu(\ol \varphi)\Big)^{4/3} d\ol A \right)^{3/2}.
\end{align*}
Now, take $\lim_{n \to \infty}$ of both sides, and use the facts that $\ol \varphi$ vanishes on $\Sigma$ and the $C^2$ convergence
of $\{\Sigma_n\}$ ensures that the mean curvature of $\Sigma_n$ in $\ol g$ is uniformly bounded as $n \to \infty$ to deduce:
\begin{align*}
 -\lim_{n \to \infty} \left(\frac{1}{16\pi} \int_{\Sigma_n} H^{4/3} dA \right)^{3/2}
        &= -\frac{1}{4}\left(\frac{1}{\pi}
           \int_{\Sigma} (\ol \nu(\ol \varphi))^{4/3} d\ol A \right)^{3/2}\\
	&=m_{\text{reg}}(\Sigma).
\end{align*}
\end{proof}
A similar result is now given for the Hawking mass; the proof also appears in \cite{robbins}.
\begin{prop}
Let $\Sigma$ be a subset of $\partial M$ consisting of regular ZAS of $g$.
If $\{\Sigma_n\}$ is a sequence of surfaces converging in $C^2$ to $\Sigma$, then
\begin{equation}
\limsup_{n \to \infty} m_H(\Sigma_n) \leq m_{\text{\emph{reg}}}(\Sigma).
\label{eq_hawking_reg}
\end{equation}
Moreover, there exists a sequence of surfaces $\{\Sigma_n^*\}$ converging in $C^2$ to $\Sigma$ such that
\begin{equation*}
\lim_{n \to \infty} m_H(\Sigma_n^*) = m_{\text{\emph{reg}}}(\Sigma).
\end{equation*}
\label{reg_mass_thm}
\end{prop}
\begin{proof}
The first part is an application of H\"{o}lder's inequality:
\begin{align*}
m_H(\Sigma_n) &= \sqrt{\frac{|\Sigma_n|_g}{16\pi}} - \frac{|\Sigma_n|_g^{1/2}}{(16\pi)^{3/2}} 
		 \int_{\Sigma_n} H^2 dA &&\text{(definition of Hawking mass)}\\
	      &\leq \sqrt{\frac{|\Sigma_n|_g}{16\pi}} - 
	\left(\frac{1}{16\pi}\int_{\Sigma_n} H^{4/3}dA\right)^{3/2} &&\text{(H\"{o}lder's inequality)}
\end{align*}
Inequality $(\ref{eq_hawking_reg})$  follows by taking $\limsup_{n \to \infty}$ and applying Definition \ref{def_zas} and Proposition
\ref{prop_H43}.

We now construct the sequence $\{\Sigma_n^*\}$. We first argue that there exists a local resolution $(\tilde g, \tilde \varphi)$
such that $\tilde \nu (\tilde \varphi) \equiv 1$ on $\Sigma$.
Let $(\ol g, \ol \varphi)$ be some local resolution defined in a neighborhood $U$ of $\Sigma$.
Let $u$ be a positive $\ol g$-harmonic function with Dirichlet boundary condition given by $\left(\ol \nu(\ol \varphi)\right)^{1/3}$ 
on $\Sigma$.  Set $\tilde g = u^4 \ol g$ and $\tilde \varphi = \frac{\ol \varphi}{u}$; we claim $(\tilde g, \tilde \varphi)$ is the desired local 
resolution.

First, note that $\tilde g$ is a smooth metric on $U$, since $u$ is positive and smooth.
Next, $\tilde\varphi^4 \tilde g = \ol \varphi^4 \ol g= g$, and $\tilde \varphi$ vanishes only on $\Sigma$.  
Now, we compute the normal derivative of $\tilde \varphi$ on $\Sigma$:
\begin{equation*}
\tilde \nu (\tilde \varphi) = u^{-2} \ol \nu \left(\frac{\ol \varphi}{u}\right) = u^{-3} \ol \nu(\ol \varphi) - \cancel{\ol \varphi \frac{\ol 
\nu(u)}{u^4}} = 1,
\end{equation*}
by the boundary condition imposed on $u$.  Thus, $(\tilde g, \tilde \varphi)$ is the desired local 
resolution.  (We remark that if $(\ol g, \ol \varphi)$ is a local harmonic resolution, then so is $(\tilde g, \tilde \varphi)$.
This follows from equation $(\ref{eq_conf_laplacian})$ in appendix B and will be used in the proof of Proposition $\ref{prop_masses_agree}$.)

Define $\Sigma_n^*$ to be the $1/n$ level set of $\tilde \varphi$, which is smooth and well-defined
for all $n$ sufficiently large.
It is clear that $\Sigma_n^*$
converges to $\Sigma$ in all $C^k$ as $n \to \infty$.

Since equality is attained in H\"{o}lder's inequality for constant functions, the proof is complete
if we show the ratio of the minimum and maximum values of $H$ (the mean curvature of $\Sigma_n^*$, measured in $g$) 
tends to 1 as $n \to \infty$.
From equation $(\ref{eq_conf_mean_curv})$ in appendix B, $H$ is given by
\begin{equation}
H = \tilde \varphi^{\,-2} \tilde H + 4 \tilde \varphi^{\,-3} \tilde \nu(\tilde \varphi),
\label{eq_H}
\end{equation}
where $\tilde H$ is the mean curvature of $\Sigma_n$ in $\tilde g$.  By $C^2$ convergence, $\tilde H$ is bounded as $n \to \infty$.  By
$C^1$ convergence, $\tilde \nu(\tilde \varphi)$ converges to 1 as $n \to \infty$. In particular, the second 
term in $(\ref{eq_H})$ dominates.  Since $\tilde \varphi$ is by definition constant on $\Sigma_n^*$,
we have proved the claim.
\end{proof}

The following corollary of Proposition \ref{reg_mass_thm}
will be pertinent when discussing the mass of arbitrary ZAS.
\begin{cor}
If $\Sigma$ is a subset of $\partial M$ consisting of regular ZAS of $g$, then
\begin{equation*}
m_{\text{\emph{reg}}}(\Sigma) = \sup_{\{\Sigma_n\}} \limsup_{n \to \infty} m_H(\Sigma_n),
\end{equation*}
where the supremum is taken over all sequences $\{\Sigma_n\}$ converging in $C^2$ to $\Sigma$.
\label{cor_reg_mass}
\end{cor}
It is necessary to take the supremum, since $\limsup_{n \to \infty} m_H(\Sigma_n)$ evidently
underestimates the regular mass in general.

To summarize, we have seen several expressions for the regular mass as the:
\begin{compactenum}
\item explicit formula (\ref{def_reg_mass}) in terms of any local resolution, 
\item limit of $-\left(\frac{1}{16\pi} \int_{\Sigma_n} H^{4/3} dA \right)^{3/2}$,
\item limit of the Hawking masses of a certain sequence of surfaces, and 
\item $\sup$ of the $\limsup$ of the Hawking masses of sequences converging to $\Sigma$.
\end{compactenum}

\subsection{The mass of arbitrary ZAS}	
\label{sec_mass_zas}
For simplicity, we assume from this point on that all components of $\Sigma := \partial 
M$ are ZAS of $g$.  We shall define only the mass of $\Sigma$ (not the mass of each 
component of $\Sigma$).
A good definition of the mass of $\Sigma$ should depend only on the local geometry near $\Sigma$
and agree with the regular mass in the case that the components of $\Sigma$ are regular (or at 
least harmonically regular).  There are three immediate candidates for the definition of the mass of 
$\Sigma$:  the supremum over all sequences 
$\{\Sigma_n\}$ converging in some $C^k$ to $\Sigma$ of
\begin{compactenum}
\item $\limsup_{n \to \infty} m_H(\Sigma_n)$ (inspired by Corollary \ref{cor_reg_mass}),
\item $\limsup_{n \to \infty} - \left(\frac{1}{16\pi}\int_{\Sigma_n} H^{4/3} dA\right)^{3/2}$ 
(inspired by Proposition \ref{prop_H43}), and
\item $\limsup_{n \to \infty} m_{\text{reg}}(\Sigma_n)$, where $\Sigma_n$ is viewed as a regular ZAS
that ``approximates'' $\Sigma$. (This is explained below; see also figure \ref{fig_approx}.)
\end{compactenum}
The first two candidates manifestly depend only on the local geometry near $\Sigma$ and agree with 
the regular 
mass for regular ZAS (by Corollary \ref{cor_reg_mass} and Proposition \ref{prop_H43}).  In fact, the 
second is greater than or equal to the first (an application of H\"{o}lder's inequality).

To explain the third quantity above, we show that each surface $\Sigma_n$ is naturally a collection of ZAS
(with respect to a new metric).  Let $\Omega_n$ be the region enclosed by $\Sigma_n$, and let
$\varphi_n$ be the unique $g$-harmonic 
function that 
vanishes on $\Sigma_n$
and tends to one at infinity.  Then $\varphi_n^4 g$ is an asymptotically flat metric on the manifold $M \sm \Omega_n$.
Moreover, $\Sigma_n$ is a collection of harmonically regular ZAS for this manifold.
This construction is demonstrated in figure \ref{fig_approx}; essentially this process approximates any ZAS $\Sigma$ with
a sequence of harmonically regular ZAS.  By construction, $(g,\varphi_n)$ is a global harmonic resolution of $\Sigma_n$, so the regular mass of 
$\Sigma_n$ is computed in this resolution as 
\begin{equation}
m_{\text{reg}}(\Sigma_n) 
	= -\frac{1}{4}\left(\frac{1}{\pi}\int_{\Sigma_n} \nu(\varphi_n)^{4/3} dA\right)^{3/2},
\label{m_reg_sigma_n}
\end{equation}
where the unit normal $\nu$ and area form $dA$ are taken with respect to $g$.  

\begin{figure}[ht]
\caption{Approximating a ZAS by harmonically regular ZAS}
\begin{center}
\includegraphics[scale=0.6]{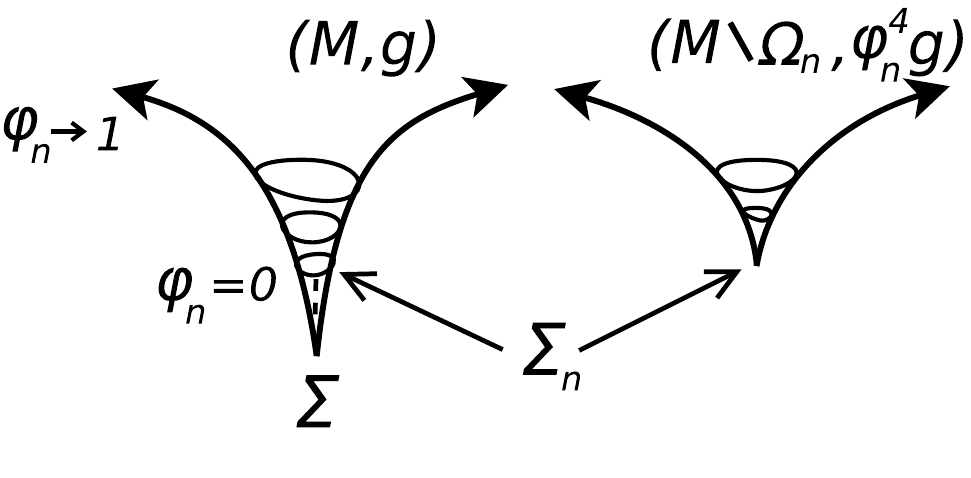}
\end{center}
\flushleft\footnotesize{On the left side, $(M,g)$ is pictured with a ZAS $\Sigma$ and a sequence $\{\Sigma_n\}$ 
of surfaces converging to it.  $\varphi_n$ is the $g$-harmonic function vanishing on $\Sigma_n$
and approaching 1 at infinity.  The conformal metric $\varphi_n^4 g$ on the region $M \sm \Omega_n$
has $\Sigma_n$ as a ZAS. (Here, $\Omega_n$ is the region enclosed by $\Sigma_n$.)}
\label{fig_approx}
\end{figure}

We adopt the third candidate for our definition of 
mass; Proposition \ref{prop_masses_agree} and Corollary \ref{cor_masses_agree}
show this definition agrees with the regular mass for harmonically regular ZAS, and
Proposition \ref{prop_local_mass_cap} shows it depends only on the local geometry near 
$\Sigma$.  The relationship between the third and first two candidates for the definition of mass 
is unknown; see section \ref{sec_alt_mass}.  Another justification for our choice of the definition 
of mass is that it naturally gives a lower bound on the ADM mass of $(M,g)$ (see Theorem 
\ref{thm_zas_ineq_full}).  It is currently unknown how to obtain such a lower bound in terms of the 
first two candidates for the mass of $\Sigma$.

\begin{definition} Let $\Sigma = \partial M$ be zero area singularities of $(M,g)$.
The \textbf{mass} of $\Sigma$ is
\begin{equation*}
m_{\text{ZAS}}(\Sigma) := \sup_{\{\Sigma_n\}} \left(\limsup_{n \to \infty} m_{\text{\emph{reg}}}(\Sigma_n)\right),
\end{equation*}
where the supremum is taken over all sequences $\{\Sigma_n\}$ converging in $C^1$ to $\Sigma$ and 
$m_{\text{reg}}(\Sigma_n)$ is given by equation $(\ref{m_reg_sigma_n})$.  
\label{def_zas_mass}
\end{definition}
Note that while the regular mass of a regular ZAS is a negative real number, 
$m_{\text{ZAS}}(\Sigma)$ takes values in $[-\infty,0]$.  In section \ref{sec_examples} we provide examples for
which $m_{\text{ZAS}}(\Sigma)=-\infty$ and $m_{\text{ZAS}}(\Sigma)=0$.  The requirement that the 
sequences $\{\Sigma_n\}$ converge in $C^1$ is explained in the proof of the following result.

\begin{prop}
If $\Sigma$ admits a global harmonic resolution, then the two definitions of mass agree.  
That is, $m_{\text{\emph{reg}}}(\Sigma) = m_{\text{ZAS}}(\Sigma)$.
\label{prop_masses_agree}
\end{prop}
\begin{proof}
Let $\Sigma$ have a global harmonic resolution $(\ol g, \ol \varphi)$. 
Let $\{\Sigma_n\}_{n=1}^\infty$ be a collection of smooth level sets
of $\ol \varphi$ that converge in $C^1$ to $\Sigma$, and let $\varphi_n$ be $g$-harmonic, vanishing on $\Sigma_n$ and tending to 1 at infinity.
Now we compute a convenient expression for the regular mass of $\Sigma_n$:
\begin{align}
m_{\text{reg}}(\Sigma_n) &= -\frac{1}{4} \left(\frac{1}{\pi} \int_{\Sigma_n} \nu(\varphi_n)^{4/3} dA\right)^{3/2}
		&&\text{(expression $(\ref{m_reg_sigma_n})$  for $m_{\text{reg}}(\Sigma_n)$)}\nonumber\\
	&= -\frac{1}{4} \left(\frac{1}{\pi} \int_{\Sigma_n} \big(\ol \varphi\, \ol \nu(\varphi_n)\big)^{4/3} d\ol A\right)^{3/2}
				&&\text{($\nu = \ol \varphi^{-2}\ol \nu$ and $dA = \ol \varphi^4 d\ol A$)}\nonumber\\
	&= -\frac{1}{4} \left(\frac{1}{\pi} \int_{\Sigma_n} \big(\ol \nu(\ol \varphi \varphi_n)\big)^{4/3} d\ol A\right)^{3/2}
		&&\text{($ \varphi_n=0$ on $\Sigma_n$).} \label{eqn_reg_mass_limit}
\end{align}
We claim that the limit $n\to \infty$ of the above equals $m_{\text{reg}}(\Sigma)$.  Let $a_n$ be
the (constant) value of $\ol \varphi$ on $\Sigma_n$. Let $E_n$ be the closure of the region exterior to $\Sigma_n$.
From formula $(\ref{eq_conf_laplacian})$ in appendix B, the function $\ol \varphi \varphi_n$ is 
$\ol g$-harmonic in $E_n$, zero on $\Sigma_n$ and 1 at infinity.
Also, $\ol \varphi - a_n$ is $\ol g$-harmonic on $M$, zero on $\Sigma_n$ and $1-a_n$ at infinity.
In particular, by the uniqueness of harmonic functions with identical boundary values, we see that
\begin{equation*}
\ol \varphi \varphi_n = \frac{1}{1-a_n}\left(\ol \varphi -a_n\right) \quad \text{ in } E_n.
\end{equation*} 
It follows that
\begin{equation}
\ol \nu (\ol \varphi \varphi_n) = \frac{1}{1-a_n}\ol \nu (\ol \varphi) \quad \text{ on } \Sigma_n.
\label{eq_m_reg_m}
\end{equation}
Continuing with equations (\ref{eqn_reg_mass_limit}), and taking $\lim_{n \to \infty}$, we have
\begin{align*}
\lim_{n \to \infty} m_{\text{reg}}(\Sigma_n) &= \lim_{n \to \infty}-\frac{1}{4} \left(\frac{1}{\pi} \int_{\Sigma_n} \left(\frac{1}{1-a_n}\ol \nu (\ol \varphi)\right)^{4/3} d\ol A\right)^{3/2} &&\text{(eqn. (\ref{eq_m_reg_m}))}\\
&= -\frac{1}{4} \left(\frac{1}{\pi} \int_{\Sigma} \left(\ol \nu (\ol \varphi)\right)^{4/3} d\ol A\right)^{3/2}\\
&= m_{\text{reg}}(\Sigma), 
\end{align*}
since $a_n \to 0$ and $\Sigma_n \to \Sigma$ in $C^1$.  Then by definition of ZAS mass,
$$m_{\text{ZAS}}(\Sigma)\geq m_{\text{reg}}(\Sigma)$$
for ZAS admitting a global harmonic resolution.  Proposition \ref{prop_reg_finite_mass} below gives 
the reverse inequality.
\end{proof}
As a consequence of this result, we may interchangeably use the terms ``mass'' and ``regular 
mass'' whenever a global harmonic resolution exists; by Proposition \ref{prop_ghr_exists},
this is the case for all harmonically regular ZAS.  Alternatively, Corollary \ref{cor_masses_agree}
directly proves (using Proposition \ref{prop_masses_agree}) that $m_{\text{ZAS}}(\Sigma)=m_{\text{reg}}(\Sigma)$ 
for harmonically regular ZAS.  The question of whether 
$m_{\text{ZAS}}(\Sigma)=m_{\text{reg}}(\Sigma)$ for merely regular ZAS 
is not fully resolved; the answer is known to be yes in the spherically symmetric case.
At the very least, we have an inequality relating the two definitions:
\begin{prop}[Proposition 56 of \cite{jauregui}]
\label{prop_reg_finite_mass}
If $\Sigma = \partial M$ consists of regular ZAS, then
$$m_{\text{ZAS}}(\Sigma) \leq m_{\text{\emph{reg}}}(\Sigma).$$
\end{prop}

We also point out that in regards to the definition of mass, there exists a sequence of 
surfaces that attains the supremum and for which the limsup may be replaced by a limit.  
\begin{prop} [Proposition 55 of \cite{jauregui}]
There exists a sequence of surfaces $\{\Sigma_n\}$
converging in $C^1$ to $\Sigma$ such that
\begin{equation*}
\lim_{n \to \infty} m_{\text{\emph{reg}}}(\Sigma_n) = m_{\text{ZAS}}(\Sigma).
\end{equation*}
\end{prop}
The proof is a basic diagonalization argument applied to a maximizing sequence of sequences $\{\Sigma_n\}_i$.

\subsection{The capacity of ZAS}
We introduce the capacity of a collection of ZAS in this section.  This quantity has a 
relationship with the mass that plays a role in the proof of the Riemannian ZAS 
inequality (Theorem \ref{thm_zas_ineq_full}).
\begin{definition}
Suppose $S=\partial \Omega$ is a surface in $M$ that is a graph over $\partial M$
(recall this terminology from section \ref{sec_surfaces}).  Let 
$\varphi$ be the unique $g$-harmonic 
function on $M \sm \Omega$ that vanishes on $S$ and tends to 1 at infinity. Then the 
\textbf{capacity} of $S$ is defined to be the number
\begin{equation*}
C(S) = \int_{M \setminus \Omega} |\nabla \varphi|^2 dV,
\end{equation*}
where $|\nabla \varphi|^2$ and $dV$ are taken with respect to $g$.
\end{definition}
The fact that $C(S)$ is finite (and moreover the existence of $\varphi$) follows from asymptotic flatness.
The above integral is unchanged if  $\varphi$ is replaced with $1-\varphi$.  Since $1-\varphi$ is harmonic,
we have \mbox{$|\nabla (1-\varphi)|^2 = \Div\big((1-\varphi) \nabla(1-\varphi)\big)$.}  Applying
Stokes' theorem and the boundary conditions on $\varphi$, we conclude that
\begin{equation}
C(S) =  \int_{S} \nu(\varphi) dA,
\label{eq_cap_boundary}
\end{equation}
where $\nu$ is the unit normal to $S$ pointing toward infinity. The capacity is also characterized
as the minimum of an energy functional:
\begin{equation}
C(S) = \inf_{f} \int_{M \sm \Omega} |\nabla f|^2 dV,
\label{eq_cap_inf}
\end{equation}
where the infimum is taken over all locally Lipschitz functions $f$ that vanish on $\Sigma$ and tend to 1 at 
infinity; $\varphi$ is the unique function attaining the infimum.

Now we recall a classical monotonicity property of capacity.
\begin{lemma}
If $S_1$ and $S_2$ are surfaces that are graphs over $\partial M$ and $S_1$ is enclosed by $S_2$,
then $C(S_1) \leq C(S_2)$.  Moreover, equality holds if and only if $S_1 = S_2$.
\label{lemma_cap}
\end{lemma}
\begin{proof}
Say $S_1 = \partial \Omega_1$ and $S_2 = \partial \Omega_2$.  Let $\varphi_1, \varphi_2$ be 
the harmonic functions vanishing on $S_1, S_2$ respectively, and tending to 1 at infinity.  
Since $\varphi_2$ can be extended continuously by zero in $\Omega_2 \sm \Omega_1$ while 
remaining locally Lipschitz, we see that $\int_{M \sm \Omega_1} 
|\nabla \varphi_2|^2 dV$ gives an upper bound for $C(S_1)$ in $(\ref{eq_cap_inf})$ yet equals 
$C(S_2)$.  Therefore $C(S_1) \leq C(S_2)$.  

In the above, if $S_1 \neq S_2$, then by the regularity of these surfaces ($C^1$ is sufficient),
the volume of $\Omega_2 \sm \Omega_1$ is positive.  The above proof shows $C(S_1) < C(S_2)$.
\end{proof}

We are ready to define the capacity of ZAS \cites{bray_npms}.

\begin{definition}
Assume the components of $\Sigma=\partial M$ are ZAS of $g$, and let 
$\{\Sigma_n\}_{n=1}^\infty$ be a sequence of surfaces converging to $\Sigma$ in $C^0$.  
Define the \textbf{capacity} of $\Sigma$ as
\begin{equation*}
C(\Sigma) = \lim_{n \to \infty} C(\Sigma_n).
\end{equation*}
The limit exists by the monotonicity guaranteed by Lemma \ref{lemma_cap}.
\end{definition}
Note that the capacity takes values in $[0, \infty)$.  We will often
distinguish between the cases of zero capacity and positive capacity. We show now 
that $C(\Sigma)$ is well-defined (as done in \cite{robbins}).
\begin{prop}
The capacity of $\Sigma$ as defined above is independent of the sequence $\{\Sigma_n\}$.
\end{prop}
\begin{proof}
Let $\{\Sigma_n\}$ and $\{\Sigma_i'\}$ be two sequences of surfaces converging to $\Sigma$ 
in $C^0$.  Then for any $n > 0$, $\Sigma_n$ encloses $\Sigma_i'$ for all $i$ sufficiently 
large.  By Lemma \ref{lemma_cap}, $C(\Sigma_i') \leq C(\Sigma_n)$
for such $n$ and $i$.  Taking the limit $i \to \infty$, we have $\lim_{i \to 
\infty}C(\Sigma_i') \leq C(\Sigma_n)$ for all $n$.  Taking the limit $n \to \infty$,
we have $\lim_{i \to \infty}C(\Sigma_i') \leq \lim_{n \to \infty}C(\Sigma_n)$.  By the 
symmetry of the argument, the opposite inequality holds as well.
\end{proof}
As an example, we show below that a collection of \emph{regular} ZAS has zero capacity. Examples of ZAS with 
positive capacity will be given in section \ref{sec_examples}.

\begin{prop}
If the components of $\Sigma=\partial M$ are regular ZAS, then the capacity of $\Sigma$ is zero.
\label{prop_regular_zas_zero_capacity}
\end{prop}
\begin{proof}
First, observe that the capacity of $\Sigma$ is equal to
\begin{equation}
C(\Sigma) = \inf_{\psi} \int_M |d \psi|_{g}^2 dV,
\label{cap_alt}
\end{equation}
where the infimum is taken over all locally Lipschitz functions $\psi$ that vanish on $\Sigma$ and tend to 1 at infinity.
Now, let $(\ol g, \ol \varphi)$ be some local resolution, and let $\ol r$ denote the distance function from $\Sigma$ with 
respect to $\ol g$.  For small $\epsilon > 0$, let $\psi_\epsilon$ be the Lipschitz test function on $M$ given by:
\begin{equation*}
\psi_\epsilon(r)=\begin{cases}
\frac{\ol r}{\epsilon}, & 0 \leq \ol r \leq \epsilon\\
1,& \text{else}
\end{cases}.
\end{equation*}
It is straightforward to show the energy of $\psi_\epsilon$ in the sense of $(\ref{cap_alt})$ is of order $\epsilon$.
\end{proof}

\subsection{The relationship between mass and capacity}
Recently appearing in the literature is an estimate relating the capacity of the boundary of an
asymptotically flat manifold to the boundary geometry \cite{bray_miao}.
In somewhat the same spirit, the following result relates the capacity of a ZAS to the Hawking masses of 
nearby surfaces \cite{robbins}.  It was proved using weak inverse mean curvature flow in the sense of Huisken
and Ilmanen \cite{imcf}.
\begin{thm}[Robbins \cite{robbins}] 
\label{thm_capacity}
Assume $(M,g)$ has nonnegative scalar curvature.  If $\Sigma$ is a 
connected ZAS with positive capacity, and if $\{\Sigma_n\}$ converges in $C^1$ to $\Sigma$, then 
\begin{equation*}
\limsup_{n \to \infty} m_H(\Sigma_n) = -\infty.
\end{equation*}
\end{thm}
With no assumption on the scalar curvature, we prove the following sufficient condition for the mass of
ZAS to equal $-\infty$.
\begin{thm}
\label{thm_pos_cap}
If $\Sigma=\partial M$ is a collection of ZAS of positive capacity, then $$m_{\text{ZAS}}(\Sigma) = -\infty.$$
\end{thm}
\begin{proof}
Suppose $\{\Sigma_n\}$ is any sequence of surfaces converging in $C^1$ to $\Sigma$.  Applying the definition of $m_{\text{reg}}$ and H\"{o}lder's 
inequality,
\begin{align*}
m_{\text{reg}}(\Sigma_n) &= -\frac{1}{4} \left(\frac{1}{\pi} \int_{\Sigma_n} \nu(\varphi_n)^{4/3} dA\right)^{3/2}\\
		  &\leq -\frac{1}{4\pi^{3/2}} |\Sigma_n|_g^{-1/2} \left(\int_{\Sigma_n} \nu(\varphi_n) dA\right)^2.
\end{align*}
The right hand side converges to $-\infty$, since $|\Sigma_n|_g \to 0$
and $\int_{\Sigma_n} \nu(\varphi_n) \to C(\Sigma) > 0$ (by
expression $(\ref{eq_cap_boundary})$ for the capacity of a surface and the definition of the 
capacity of a ZAS).  Therefore $\limsup_{n \to \infty} m_{\text{reg}}(\Sigma_n) = -\infty$ for arbitrary $\{\Sigma_n\}$, so $m_{\text{ZAS}}(\Sigma) = -\infty$.
\end{proof}
We show in section \ref{sec_examples} that the converse fails; there exist ZAS of zero 
capacity yet negative infinite mass.

\subsection{The local nature of mass and capacity}
A satisfactory definition of the mass of a collection of ZAS ought to only depend on the local 
geometry near the singularities.  Here we establish that the mass and the (sign of) capacity of ZAS are 
inherently local notions, despite their definitions in terms of global geometry. This section
is meant only to illustrate these ideas and is not essential to the main theorems.
Most of the proof of the following proposition was given by Robbins in \cite{robbins}.
\begin{prop}
Suppose all components of $\Sigma = \partial M$ are ZAS of $g$, and let $U$ be any neighborhood of $\Sigma$.  
Then 1) the sign of the capacity of $\Sigma$ and 2) the mass of $\Sigma$
depend only on the restriction of $g$ to $U$.
\label{prop_local_mass_cap}
\end{prop}
\begin{proof}
First, assume $\Sigma$ has zero capacity with respect to $g$.
Let $S$ be some surface in $M$ enclosing $\Sigma$ such that $S$ and the region it bounds are contained in $U$.
Let $\Sigma_n$ converge to $\Sigma$ in $C^0$.  By 
truncating finitely many terms of the sequence,
we may assume all $\Sigma_n$ are enclosed by $S$.
Let $\varphi_n$ be $g$-harmonic, equal to $0$ on $\Sigma_n$, and tending to 1 at infinity.  Let $\epsilon_n$ be the 
minimum value attained by $\varphi_n$ on $S$.  Let $f_n^-$ and $f_n^+$ 
be functions in the region bounded by $\Sigma_n$ and $S$ that are
$g$-harmonic, equal to  $0$ on $\Sigma_n$, with $f_n^-|_S = \epsilon_n$ 
and $f_n^+|_S = 1$.  The setup is illustrated in figure \ref{fig_local_capacity}.
\begin{figure}[ht]
\caption{Functions in the proof of Proposition \ref{prop_local_mass_cap}}
\begin{center}
\includegraphics[scale=0.6]{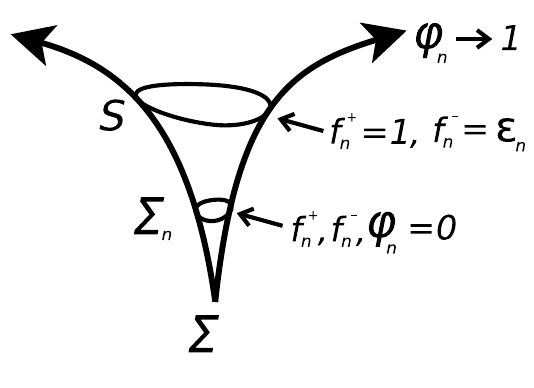}
\end{center}
\flushleft\footnotesize{This diagram illustrates the boundary values of the harmonic functions used in the 
proof of Proposition \ref{prop_local_mass_cap}.}
\label{fig_local_capacity}
\end{figure}
By the maximum 
principle, the following inequalities hold on $\Sigma_n$:
\begin{equation}
0< \nu (f_n^-) \leq \nu(\varphi_n) \leq \nu(f_n^+),
\label{capacity_ineq}
\end{equation}
where $\nu$ is the outward unit normal to $\Sigma_n$ with respect to $g$.  Integrating the first pair of inequalities over $\Sigma_n$
and using expression $(\ref{eq_cap_boundary})$ for capacity, we have
\begin{equation*}
0 < \int_{\Sigma_n} \nu(f_n^-) dA \leq C(\Sigma_n).
\end{equation*}
By assumption, $C(\Sigma_n) \to 0$ as $n \to \infty$, so $\int_{\Sigma_n} \nu(f_n^-) dA \to 0$.
But by the uniqueness of harmonic functions with identical boundary values, $f_n^- = \epsilon_n f_n^+$, so \mbox{$\epsilon_n \int_{\Sigma_n} 
\nu(f_n^+) dA 
\to 0$.}  Now, $\{\epsilon_n\}$ is an increasing sequence
by the maximum principle, so it must be that \mbox{$\int_{\Sigma_n} \nu(f_n^+) dA \to 0$.}  By integrating the first and last inequality in 
$(\ref{capacity_ineq})$
over $\Sigma_n$, we have
\begin{equation}
0 < C(\Sigma_n) \leq \int_{\Sigma_n} \nu(f_n^+) dA.
\end{equation}
We asserted that the right hand side converges to zero (implying $C(\Sigma)=0$), and this fact depends only on the data $(U,g|_{U})$.
Thus, the property of zero capacity is determined by the restriction of $g$ to $U$. This completes the first part of the proof.

For the second part, if $\Sigma$ has positive capacity, then by Theorem \ref{thm_pos_cap}, 
the mass of $\Sigma$ is $-\infty$.
However, the property of positive capacity is determined by $g|_{U}$, so the condition $m_{\text{ZAS}}(\Sigma)=-\infty$
is also determined by $g|_{U}$.  Now, we may assume $\Sigma$ has zero capacity.

From inequalities $(\ref{capacity_ineq})$ and $f_n^- = \epsilon_n f_n^+$, we have
\begin{equation}
\epsilon_n^{4/3}\int_{\Sigma_n} \nu (f_n^+)^{4/3}dA  \leq \int_{\Sigma_n} \nu (\varphi_n)^{4/3} dA 
	\leq \int_{\Sigma_n} \nu(f_n^+)^{4/3} dA.
\label{capacity_ineq4}
\end{equation}
By Lemma \ref{capacity_lemma} (below) and the fact that $\Sigma$ has zero
capacity, the sequence of numbers \{$\epsilon_n$\} converges to 1.  Taking the $\limsup$ of
$(\ref{capacity_ineq4})$, we have
\begin{equation}
\limsup_{n \to \infty} \int_{\Sigma_n} \nu (\varphi_n)^{4/3}dA = \limsup_{n \to \infty} \int_{\Sigma_n} \nu (f_n^+)^{4/3}dA.
\end{equation}
Since the right hand side depends only on $g|_{U}$,
it now follows from the definition of the mass of $\Sigma$ that $m_{\text{ZAS}}(\Sigma)$ depends only on $g|_{U}$.
\end{proof}
The following lemma will also be used in the proof of Theorem \ref{thm_zas_ineq_full}. 
\begin{lemma}
Suppose $\Sigma=\partial M$ consists of ZAS and has zero capacity. If $\{\Sigma_n\}$ is a sequence of surfaces 
converging in $C^0$ to $\Sigma$ and $\varphi_n$ is the harmonic 
function
vanishing on $\Sigma_n$ and tending to 1 at infinity, then $\varphi_n(x) \to 1$ 
pointwise in $M \sm \Sigma$, with uniform convergence on compact
subsets.  
\label{capacity_lemma}
\end{lemma}
\begin{proof}
Let $K \subset\subset M \sm \Sigma$. Then for all $n$ sufficiently large, $K$ is 
contained in the region exterior
to $\Sigma_n$.  Without loss of generality, assume $\Sigma_n$ encloses $\Sigma_{n+1}$ 
for all $n$.  Then by the maximum principle, $\{\varphi_n\}$
is an increasing sequence of harmonic functions on $K$, each bounded between $0$ and $1$ and tending to 1 at infinity.
Thus, $\{\varphi_n\}$ converges uniformly on $K$ (indeed, in any $C^k$) to a harmonic function $\varphi$.  Since 
$\Sigma$ has zero capacity, $\lim_{n \to \infty} \int_{K} 
|\nabla \varphi_n|^2 dV = 0$, so $\int_K |\nabla \varphi|^2 dV=0$.  So $\varphi$ is constant on $M \sm \Sigma$
and must identically equal 1.
\end{proof}
Having shown that the ZAS mass is a local notion, we have the following corollary, which improves
Proposition \ref{prop_masses_agree}.
\begin{cor}
If $\Sigma=\partial M$ consists of harmonically regular ZAS, then 
$m_{\text{\emph{reg}}}(\Sigma) = m_{\text{ZAS}}(\Sigma)$.
\label{cor_masses_agree}
\end{cor}
\begin{proof}
The proof of Proposition \ref{prop_local_mass_cap} shows that for the purposes of computing the mass
of $\Sigma$, the functions $\{f_n^+\}_{n=1}^\infty$ may be used in lieu of $\{\varphi_n\}_{n=1}^\infty$ 
(where $f_n^+$ is harmonic, $0$ on $\Sigma_n$ and 1 on a fixed surface $S$).  This was pointed out in \cite{robbins}. Choose $S=\partial \Omega$
so that $\Omega$ is contained in the domain of a local harmonic resolution, and proceed as in Proposition
\ref{prop_masses_agree}.
\end{proof}
Thus, we may interchangeably use the terms ``mass'' and ``regular 
mass'' for harmonically regular ZAS.

\section{Examples: spherically symmetric ZAS and removable singularities}
\label{sec_examples}
This section involves many computational details; some readers may wish to
skip to section \ref{sec_main}, which includes the main theorems.   Many of the calculations below were also carried out by Robbins \cite{robbins}.

Here, we consider $M=S^2 \times [0, \infty)$ with a spherically 
symmetric metric $g$ given by
\begin{equation}
ds^2 = dr^2 + \frac{A(r)}{4\pi} d\sigma^2,
\label{metric_spher_symm}
\end{equation}
where $r$ is the geodesic distance along the axis of symmetry, $A(r)$ is smooth and positive for $r > 0$ and extends continuously to $0$ at $r=0$,
and $d\sigma^2$ is the round metric on the sphere of radius 1.  By construction, the value of $A(r)$ equals the area of the sphere $S^2 \times 
\{r\}$.
Assume that $g$ is asymptotically flat, so that $A(r)$ is asymptotic to $4\pi r^2$. (In particular, $\int_{1}^\infty \frac{dr}{A(r)} < \infty$.)
By Proposition \ref{prop_zas_criteria}, $\Sigma = S^2 \times \{0\}$ is a ZAS of $g$.  

\subsection{Capacity}
We now derive an explicit expression for the capacity of $\Sigma$.
The Laplacian with respect to $g$ acting on functions of $r$ is:
\begin{equation}
\Delta f(r) = f''(r) + \frac{A'(r)}{A(r)} f'(r).
\label{eq_sph_symm_laplacian}
\end{equation}
It follows that the harmonic function vanishing on $S_\rho:=S^2 \times \{\rho\}$ (with $\rho > 0$) and tending to 1 at infinity is given by
\begin{equation*}
\varphi_\rho(r) = \frac{1}{\int_\rho^\infty \frac{dr}{A(r)}} \int_\rho^r \frac{dr}{A(r)}.
\end{equation*}
The improper integral is finite, by the asymptotic behavior of $A(r)$.  Now, the capacity of $S_\rho$ is
\begin{align}
C(S_\rho) &= \int_{S_\rho} \nu (\varphi_\rho) dA &&\text{(expression $(\ref{eq_cap_boundary})$ for capacity)}\nonumber\\
	  &= \int_{S_\rho} \frac{1}{\int_\rho^\infty \frac{dr}{A(r)}} \frac{1}{A(\rho)} dA
		&&\left(\nu = \frac{\partial}{\partial r}\right)\nonumber\\
	  &= \left(\int_{\rho}^\infty \frac{dr}{A(r)}\right)^{-1}
		&&\text{($|S_\rho|_g = A(\rho)$)}\label{eq_sphr_symm_cap}.
\end{align}
In particular, the capacity of $\Sigma$ is given by
\begin{equation*}
C(\Sigma) = \lim_{\rho \to 0^+} C(S_\rho) = \left(\int_0^\infty \frac{dr}{A(r)}\right)^{-1}.
\end{equation*}
Thus, we see that the capacity of $\Sigma$ is zero if and only if the improper integral $\int_0^\infty \frac{dr}{A(r)}$ is infinite (which
holds if and only if $\int_0^\epsilon \frac{dr}{A(r)}$ is infinite for all $\epsilon > 0$).  Otherwise, the
capacity is positive and finite.  

\subsection{Explicit examples}
In this section assume $A(r)$ is given by
\begin{equation}
A(r) = 4\pi r^\alpha, \qquad \text{for } 0< r \leq 1,
\label{eq_ar}
\end{equation}
where $\alpha > 0$ is a constant.  We need only define $A(r)$ on an interval, since the examples of this section 
are purely local.  Our goal is to fill in Table \ref{table_examples}, which shows the mass and capacity of 
the ZAS of this metric for the various values of $\alpha$, as well as whether or not each ZAS is regular or 
removable (explained below).

\begin{table}
\caption{ZAS data for metrics of the form $ds^2=dr^2+r^\alpha d\sigma^2$}
\label{table_examples}
\begin{tabular}{@{}lllll@{}} \toprule
range of $\alpha$ & capacity of $\Sigma$ & mass of $\Sigma$ & regular ZAS & 
 removable singularity\\  \midrule
$0 < \alpha < 1$ & positive & $-\infty$ & no &no\\
$1 \leq \alpha <\frac{4}{3}$ & zero &$-\infty$& no &no\\
$\alpha =\frac{4}{3}$& zero &$-\frac{2}{9}$& yes &no\\
$\frac{4}{3} < \alpha < 2$& zero & zero & no &no\\
$\alpha=2$& zero & zero & no &yes\\
$2 < \alpha < \infty$& zero & zero & no &no\\ \bottomrule
\end{tabular}\\
\flushleft\footnotesize{For each $\alpha > 0$, the spherically symmetric metric $ds^2=dr^2 + r^\alpha d\sigma^2$ has a ZAS at $r=0$.
The above table gives the following properties of this ZAS for each possible value of $\alpha$: the sign of the capacity, the mass,
whether or not the ZAS is regular, and whether or not the ZAS is a removable singularity.}
\end{table}

By the above discussion of capacity, we see that $C(\Sigma)>0$ if and only if $\alpha < 1$.
By Theorem \ref{thm_pos_cap}, if $0< \alpha < 1$, then $m_{\text{ZAS}}(\Sigma) = -\infty$.  We now determine
the mass of $\Sigma$ assuming $\alpha \geq 1$.  The first step is to compute the regular masses of the concentric spheres $S_\rho$:
\begin{align}
m_{\text{reg}}(S_\rho) &= - \frac{1}{4}\left( \frac{1}{\pi} \int_{S_\rho} \nu(\varphi_\rho)^{4/3}dA\right)^{3/2}\nonumber\\
	&= - \frac{1}{4\pi^{3/2}}\left(\int_{S_\rho} \left(\frac{1}{\int_\rho^\infty 
		\frac{dr}{A(r)}} \frac{1}{A(\rho)}\right)^{4/3}dA\right)^{3/2}\nonumber\\
	&= - \frac{1}{4\pi^{3/2}}\left(\frac{A(\rho)^{-1/3}}{\left(\int_\rho^\infty 
		\frac{dr}{A(r)}\right)^{4/3}}\right)^{3/2}\nonumber\\
	&= - \frac{1}{4\pi^{3/2}}\left(\frac{A(\rho)^{-1/4}}{\int_\rho^\infty 
		\frac{dr}{A(r)}}\right)^{2}. \label{eq_sph_symm_cap}
\end{align}
We can use L'Hopital's rule to evaluate this limit (of the form $\frac{\infty}{\infty}$) as $\rho \to 0$:
\begin{align}
\lim_{\rho \to 0+} m_{\text{reg}}(S_\rho)
	&= - \frac{1}{4\pi^{3/2}}\left(\lim_{\rho \to 0^+}\frac{-\frac{1}{4}A(\rho)^{-5/4}A'(\rho)}{ 
		-\frac{1}{A(\rho)}}\right)^{2}\nonumber\\
	&= - \frac{1}{64\pi^{3/2}}\left(\lim_{\rho \to 0^+}A(\rho)^{-1/4}A'(\rho)\right)^{2}\nonumber\\
	&= - \frac{1}{64\pi^{3/2}}\left(\lim_{\rho \to 
		0^+}(4\pi)^{3/4}\alpha\rho^{3\alpha/4-1}\right)^{2}.\label{m_reg_limit}
\end{align}
This limit is $-\infty$ if $1 \leq \alpha < 4/3$, finite but nonzero if $\alpha = 4/3$, and zero if $\alpha > 4/3$.  For $\alpha > 4/3$,
this shows that the mass of $\Sigma$ is zero.  (It is at least zero by our computation, but it is a priori at most zero.)  For $1 \leq \alpha < 
4/3$, we claim the mass of $\Sigma$ is $-\infty$.

Given any sequence of surfaces $\{\Sigma_n\}$ converging in $C^1$ to $\Sigma$, we compare 
them to a sequence of round spheres. Let $\rho_n>0$ be the minimum value
of the $r$-coordinate that $\Sigma_n$ attains; then $\Sigma_n$ encloses $S_{\rho_n}$, so 
$C(\Sigma_n) 
\geq C(S_{\rho_n})$.  Also, since
$\Sigma_n$  converges in $C^1$, the ratio of areas $a_n := \frac{|\Sigma_n|}{A(\rho_n)}$
converges to 1.
Using the proof of
Theorem \ref{thm_pos_cap}, we can estimate the regular mass of $\Sigma_n$:
\begin{align*}
m_{\text{reg}}(\Sigma_n) &\leq -\frac{1}{4\pi^{3/2}} |\Sigma_n|_g^{-1/2}\, C(\Sigma_n)^2\\
	&= -\frac{1}{4\pi^{3/2}}a_n^{-1/2}A(\rho_n)^{-1/2} C(S_{\rho_n})^2\\
	&= a_n^{-1/2} m_{\text{reg}}(S_{\rho_n}),
\end{align*}
where the last equality comes from equations $(\ref{eq_sphr_symm_cap})$ and $(\ref{eq_sph_symm_cap})$.
Taking $\limsup_{n \to \infty}$ of both sides (and using $\rho_n \to 0$), we have
$\limsup_{n \to \infty} m_{\text{reg}}(\Sigma_n) = -\infty$.  (We showed above that this limit 
is $-\infty$ for concentric round spheres.)
Thus, $m_{\text{ZAS}}(\Sigma) = -\infty$ if $1 \leq \alpha < 4/3$.

For $\alpha=4/3$,  we show below that $\Sigma$ is harmonically regular.  By Corollary \ref{cor_masses_agree}, its mass is 
given by $(\ref{m_reg_limit})$, which evaluates to $-\frac{2}{9}$.  We now determine the values of $\alpha$
for which $\Sigma$ is a regular ZAS.
\begin{lemma}
The ZAS $\Sigma$ of the metric $ds^2 = dr^2 + \frac{A(r)}{4\pi} d\sigma^2$ with $A(r)$ given by
$(\ref{eq_ar})$ is regular if and only if $\alpha = 4/3$.  In this case, $\Sigma$ is harmonically 
regular.
\end{lemma}
\begin{proof}
We find a necessary condition for $\Sigma$ to be regular.  Assume $(\ol 
g, \ol \varphi)$ is some local resolution; by spherical symmetry we may assume $\ol \varphi$ and $\ol g$ depend 
only on $r$.  By applying a conformal transformation, we may assume that the spheres $S^2\times 
\{\rho\}$ have constant area $4\pi$ in $\ol g$ (i.e., $\ol \varphi^{-4}(r) 
A(r) \equiv 4\pi$).  (In other words, the metric $\ol g$ is that of a round cylinder
with spherical cross sections.)
Then
\begin{equation*}
\ol \varphi(r) =  \left(\frac{A(r)}{4\pi}\right)^{1/4} = r^{\alpha/4}.
\end{equation*}
A necessary condition for a local resolution is that $\ol \nu(\ol \varphi)$ is positive and finite 
on $\Sigma$.  Compute this normal derivative:
\begin{equation*}
\ol\nu(\ol \varphi) =\ol \varphi^2 \frac{\partial}{\partial r}\ol \varphi(r)
	=\frac{\alpha}{4}r^{3\alpha/4-1}.
\end{equation*}
This is positive and finite in the limit $r \to 0$ if and only if $\alpha = 4/3$.  Thus, 
$\alpha=4/3$ is necessary for the existence of a local resolution.

If $\alpha=4/3$, then $\ol \varphi(r) = r^{1/3}$.  A calculation shows that the arc length parameter 
for $\ol g$ is given as $\ol r = 3r^{1/3}$.  Thus, $\ol \varphi(\ol r) = \frac{1}{3}\ol r$, which 
is smooth on $[0, 1)$; moreover, $\ol \varphi$ is $\ol g$-harmonic.  Thus, $(\ol g, \ol \varphi)$ 
is a local 
harmonic resolution.
\end{proof}
Evidently the case $\alpha = 4/3$ is special.  It is left to the reader to show that if
the Schwarzschild ZAS metric is written in the form $(\ref{metric_spher_symm})$, then to first order,
$A(s)$ is given by a constant times $s^{4/3}$, where $s$ is the distance to $\Sigma$.

\subsection{Removable singularities}
\label{sec_removable}
For our purposes a ``removable singularity'' is a point deleted from the interior of a smooth Riemannian manifold.  Such
singularities can be viewed as ZAS.  For example
the manifold $\R^3 \sm \{0\}$ with the flat metric $\delta$ has its interior ``boundary'' as a
ZAS.  To see this, let $M = S^2 \times [0,\infty)$
with the flat metric $ds^2 = dr^2 + r^2 d\sigma^2$ on its interior, which is isometric to $\R^3 \sm \{0\}$.
Clearly the boundary $\Sigma := S^2 \times \{0\}$ is a ZAS.  It is not difficult to see that 
the ZAS $\Sigma$ of the spherically symmetric metric $(\ref{eq_ar})$ is removable if and only if $\alpha=2$, which
corresponds to a deleted point in flat space.

In fact, it is straightforward to show that a collection of removable singularities has zero mass; see \cite{jauregui}.
(The key tool is the existence of a harmonic function that blows up at the removable singularities.) However, not all 
ZAS with zero mass are removable, as seen in earlier examples.

There is also a notion of a ``removable $S^1$ singularity''; this occurs when an embedded circle is removed from the interior
of a smooth Riemannian manifold.  The resulting space has a zero area singularity that is topologically a 2-torus.  It is
left to the reader to verify that removable $S^1$ singularities have mass equal to $-\infty$ (at least for a circle deleted
from $\R^3$).  Thus, our definition of mass is not well-adapted to studying these types of singularities.

\subsection{A regular ZAS that is not harmonically regular}
In this section, we prove by example that there exists a 
ZAS that has a local resolution but no local harmonic resolution; we begin by showing that the problem of finding a local harmonic 
resolution is equivalent to solving a linear elliptic PDE. 
Let $(\ol g, \ol \varphi)$ be some local resolution of a ZAS $\Sigma^0$.
Suppose for now that there exists
a local harmonic resolution $(\tilde g, \tilde \varphi)$ of $\Sigma^0$ defined on a 
neighborhood $U$.  Set
$u=\frac{\ol \varphi}{\tilde \varphi}$, which is smooth and positive on $M$ by Lemma
\ref{lemma_smooth_extension}.  Also, $\tilde g = u^4 \ol g$.  Apply
formula $(\ref{eq_conf_laplacian})$ from appendix B (with $\tilde \varphi$ playing
the role of $\phi$).  This leads to the equation
\begin{equation*}
\ol \Delta \ol \varphi = \tilde \varphi \ol \Delta \left(\frac{\ol \varphi}{\tilde 
\varphi}\right),
\end{equation*}
since $\tilde \varphi$ is assumed to be $\tilde g$-harmonic.  
Then $u$ satisfies $\ol \Delta u = fu$, where $f = \frac{\ol \Delta \ol\varphi}{\ol 
\varphi}$.  Then $u$ satisfies the linear elliptic PDE
\begin{equation}
\begin{cases}
Lu=0,& \text{in } U\\
			  u >0,& \text{on } \Sigma^0
\end{cases}
\label{eq_local_harmonic_res}
\end{equation}
where $L:= \ol \Delta - f$.  On the other hand, if $u$ is some positive solution to 
$(\ref{eq_local_harmonic_res})$ that is smooth up to and including the boundary,
then the above discussion implies that $(\tilde g, \tilde \varphi)$ is a local
harmonic resolution of $\Sigma^0$, where $\tilde \varphi := \frac{\ol \varphi}{u}$ and $\tilde 
g := u^4 \ol g$.

Note that solutions of $(\ref{eq_local_harmonic_res})$ may lose regularity at the boundary if, for
some local resolution $(\ol g, \ol \varphi)$, the function $\frac{\Delta \ol \varphi}{\ol \varphi}$ is not smooth 
at the boundary.  Now we exhibit such an example.  

Consider a spherically symmetric metric $\ol g$ given by $ds^2 = dr^2 + e^{r}d\sigma^2$ and a 
function $\ol\varphi(r)=r$, both for $0\leq r \leq 1$.  Then $(\ol g, \ol \varphi)$ is a local 
resolution of the ZAS of $g:= \ol\varphi^4 \ol g$.  We assume that if a local harmonic 
resolution exists, it is also spherically symmetric.  By equation $(\ref{eq_sph_symm_laplacian})$,
the Laplacian on functions of $r$ is given by
\begin{equation*}
\ol\Delta \psi(r) = \psi''(r) + \psi'(r).
\end{equation*}
Thus, for our choice of $\ol \varphi(r)$, we have $\frac{\ol \Delta \ol \varphi}{\ol 
\varphi}=\frac{1}{r}$, so equation $(\ref{eq_local_harmonic_res})$ becomes
\begin{equation*}
u''(r) + u'(r) - \frac{u(r)}{r} = 0,
\end{equation*}
where $u$ is the quotient of $\ol \varphi$ and the unknown harmonic resolution function.
The general solution of this second order linear ODE on $(0,1)$ is given by
\begin{equation*}
u(r)=C_1 r + C_2\left(-e^{-r} + r \int_r^\infty \frac{e^{-t}}{t} dt\right)
\end{equation*}
where $C_1$ and $C_2$ are arbitrary constants.  To satisfy the condition $u(0)>0$, it must be 
that $C_2 \neq 0$.  However, in this case, $u'(r)$ does not extend smoothly to zero:
\begin{equation*}
u'(r) = C_1 + C_2\int_r^\infty \frac{e^{-t}}{t}dt,
\end{equation*}
which diverges at $r=0$. Thus, there exists no solution of $(\ref{eq_local_harmonic_res})$ for this 
choice of $g$ obeying the necessary boundary conditions and extending smoothly to the boundary.  
In other words, the ZAS $\Sigma$ admits no local harmonic resolution.

\section{The Riemannian ZAS inequality}
\label{sec_main}
In this section we prove the two main theorems of this paper, stated below.
Both were introduced and proved in \cite{bray_npms} (except the case of equality in Theorem \ref{thm_zas_ineq}, which is a new
result).
The proofs rely on a certain unproven conjecture, explained in
section \ref{sec_conf_conj}.

\subsection{The main theorems}
\begin{thm}[Riemannian ZAS inequality, harmonically regular case]
Suppose $g$ is an asymptotically flat metric on $M \sm \partial M$ of nonnegative scalar curvature such that
all components of the boundary $\Sigma=\partial M$ are ZAS.  Assume there exists a global 
harmonic resolution $(\ol g, \ol \varphi)$ of $\Sigma$. Also assume Conjecture \ref{conj_conformal} (below) holds.
Then the ADM mass $m$ of $(M,g)$ satisfies
\begin{equation}
m \geq m_{\text{ZAS}}(\Sigma),
\label{zas_ineq}
\end{equation}
where $m_{\text{ZAS}}(\Sigma)$ is as in Definition \ref{def_zas_mass}.
Equality holds in $(\ref{zas_ineq})$ if and only if $(M,g)$ is a Schwarzschild ZAS of mass $m<0$.
\label{thm_zas_ineq}
\end{thm}
If $\Sigma$ has components $\{\Sigma_i\}_{i=1}^k$ with 
respective regular masses $m_i$, then by equation $(\ref{eq_add_masses})$ and Proposition 
\ref{prop_masses_agree}, 
the statement $m \geq m_{\text{ZAS}}(\Sigma)$ can be written 
\begin{equation*}
m \geq - \left(\sum_{i=1}^k |m_i|^{2/3}\right)^{3/2}.
\end{equation*}

\begin{thm}[Riemannian ZAS inequality, general case]
Suppose $g$ is an asymptotically flat metric on $M \sm \partial M$ of nonnegative scalar curvature such that
all components of the boundary $\Sigma=\partial M$ are ZAS.%, and that $M$ has no other singularities. 
Assume that Conjecture $\ref{conj_conformal}$ holds.
Then the ADM mass $m$ of $(M,g)$ satisfies $m \geq m_{\text{ZAS}}(\Sigma)$.
\label{thm_zas_ineq_full}
\end{thm}
We have not attempted to characterize the case of equality in the general case; this issue 
is discussed in section \ref{sec_coe_general}. 

\noindent\textbf{Remarks}
\begin{compactenum}
\item Robbins \cites{robbins, robbins_paper} showed Theorem \ref{thm_zas_ineq_full} in the case $\partial M$ is connected by 
applying the technique of weak inverse mean curvature flow (IMCF) due to Huisken
and Ilmanen \cite{imcf}.  However, this approach gives no bound on the ADM mass if more than 
one ZAS is present; the proof relies crucially on the monotonicity of the Hawking mass under 
IMCF, but monotonicity is lost when the surfaces travel past a ZAS.

\item Lam proved a version of Theorem \ref{thm_zas_ineq_full} for manifolds that arise as graphs
in Minkowski space \cite{lam}.

\item Jauregui proved a weakened version of the ZAS inequality for manifolds that
are conformally flat and topologically $\R^n$ minus a finite number of domains \cite{jauregui_conf_flat}.  The non-sharpness
manifests as a multiplicative error term involving the isoperimetric ratio of the domains.

\item Theorem \ref{thm_zas_ineq} and the definition of ZAS mass have alternate interpretations in terms of
certain invariants of the harmonic conformal class of asymptotically flat metrics without singularities; see
\cite{jauregui_hci}.

\item Note that a manifold containing ZAS typically cannot be extended to a smooth, complete manifold;
consequently the positive mass theorem does not apply to such spaces.
\end{compactenum}

The following special case with $m_{\text{ZAS}}(\Sigma)=0$ deserves attention, since it is a new 
version of the positive mass theorem that allows for certain types of singularities and incomplete metrics.
\begin{cor} [Positive mass theorem with singularities]
Suppose $(M,g)$ satisfies the hypotheses of Theorem \ref{thm_zas_ineq_full} with $m_{\text{ZAS}}(\Sigma)=0$.  Then $m \geq 0$,
where $m$ is the ADM mass of $(M,g)$.
\end{cor}
The above result is a nontrivial statement, since in section \ref{sec_examples} that there exist 
singularities of zero mass that are not removable.  
\vspace{1mm}

\subsection{A conjecture in conformal geometry}
\label{sec_conf_conj}
Our main tool in the proof of Theorem \ref{thm_zas_ineq} is the Riemannian Penrose inequality.  
The latter applies to manifolds whose boundary is a minimal surface, so it
will be necessary to transform our manifold into such (by a conformal change, for example).  This motivates 
the following problem, independent of ZAS theory.

\vspace{0.15in}
\noindent\textbf{Problem:}
Given a smooth asymptotically flat manifold $(N,h)$ with 
compact boundary (and no singularities), find a positive harmonic function $u$ such that the boundary has 
zero mean curvature in the metric $u^4 h$ and $u$ approaches a constant at infinity.
(We require $u$ to be harmonic so that the sign of scalar curvature is preserved.)  
\vspace{0.15in}

Unfortunately, this problem is unsolvable in general.  Proposition 6 of \cite{jauregui} gives the following
geometric obstruction: if mean curvature $H$ of the boundary of $(N,h)$ exceeds $4\nu(\varphi)$ 
(where $\varphi$ is harmonic with respect to $h$, 0 on $\Sigma$ and 1 at infinity, and $\nu$ is the inward
point normal to the boundary), then no such harmonic function $u$ exists so that
$u^4h$ has zero mean curvature boundary.  An adaptation of the argument shows that this procedure still
fails even if $u$ is allowed to be superharmonic.  A topological obstruction to solving the above
problem can be found in \cite{bray_npms}.

Despite these obstructions, we predict in the following conjecture that
we can ``almost'' solve the above problem in the following sense: given $(N,h)$ as above, 
there exists a positive $h$-harmonic function $u$ such that $N$ contains a compact surface 
$\tilde \Sigma$ of zero mean curvature (with respect to $u^4h$) such that $\tilde \Sigma$ and $\Sigma$ 
the same area (with respect to $u^4h$).
The following conjecture was first given in \cite{bray_npms}.
(For details on the outermost minimal area enclosure of the boundary of an asymptotically flat 
manifold, see section \ref{sec_omae} in appendix A.)
\begin{conj}[``Conformal conjecture'']
Let $(N,h)$ be a smooth asymptotically flat 3-manifold with compact, smooth, nonempty boundary 
$\Sigma$, with $h$ extending smoothly to $\Sigma$.
There exists a smooth, positive function $u$ on $N$ and metric
$h_0 = u^4 h$ satisfying the following conditions.  Let $\tilde \Sigma$ be the outermost 
minimal area enclosure of $\Sigma$ with respect to $h_0$.
\begin{enumerate}
\item $u$ is harmonic with respect to $h$ and tends to 1 at infinity.
\item In the metric $h_0$, the areas of $\tilde \Sigma$ and $\Sigma$ are equal.
\item $\tilde \Sigma$ is smooth with zero mean curvature.
\end{enumerate}
\label{conj_conformal}
\end{conj}
See section \ref{sec_omae}
in appendix A for a discussion of outermost minimal area enclosures. 
We remark that the conjecture is known to be true in the spherically symmetric case.
The effect of applying the Conjecture \ref{conj_conformal} to an asymptotically flat 
manifold with boundary is illustrated in figure \ref{fig_conf_conj}.

\begin{figure}[ht]
\caption{Conjecture \ref{conj_conformal}, illustrated}
\begin{center}
\includegraphics[scale=0.35]{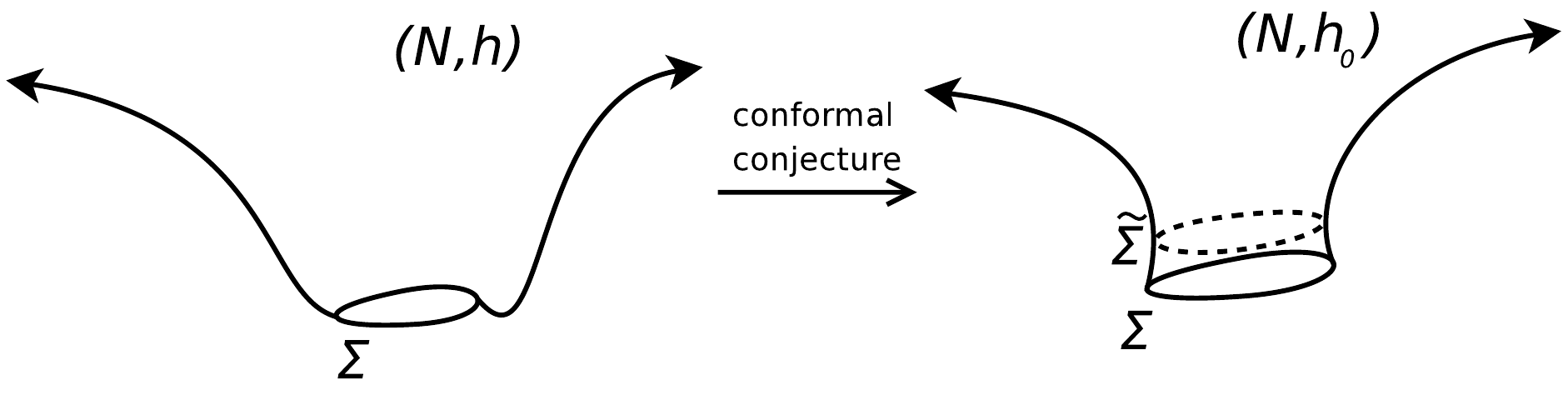}
\end{center}
\flushleft\footnotesize{If the conformal conjecture applies to $(N,h)$ as above on the left, then it gives 
the existence of $(N,h_0)$ on the right. For the latter manifold, the outermost minimal area 
enclosure, $\tilde \Sigma$, is a zero mean curvature surface and its area equals that of $\Sigma$.}
\label{fig_conf_conj}
\end{figure}

Now we explain how this conjecture relates to Theorems \ref{thm_zas_ineq} and 
\ref{thm_zas_ineq_full}.  We wish to obtain a lower
bound on the ADM mass of $(M,g)$.  For the first theorem, our strategy is to take a global 
harmonic resolution and apply Conjecture \ref{conj_conformal};
this produces a manifold-with-boundary to which the Riemannian Penrose inequality applies.  
(Note that because the conformal factor is harmonic,
the property of nonnegative scalar  curvature is preserved.  See equation 
$(\ref{eq_conf_scalar_curv})$ in appendix B.) This gives a lower bound on the ADM mass of the 
latter manifold; in the next section we explain how to transform it to the desired bound.  Finally,
Theorem \ref{thm_zas_ineq_full} will be a consequence of Theorem \ref{thm_zas_ineq} and the definition
of mass.

If true, Conjecture \ref{conj_conformal} implies the following statement, which is better adapted for 
the proof of Theorem \ref{thm_zas_ineq}.
\begin{conj}
Let $(M,g)$ have boundary $\Sigma$ consisting of ZAS and admitting a global harmonic resolution of $\Sigma$.
There exists a global harmonic resolution $(\ol g_0, \ol \varphi_0)$ of $\Sigma$ such that 
$\tilde{ 
\Sigma}$ (the outermost minimal area enclosure of $\Sigma$ in the metric $\ol g$)  is minimal and 
$|\Sigma|_{\ol g} = |\tilde { 
\Sigma}|_{\ol g}$.
\label{conformal_conj_cor}
\end{conj}
\begin{proof}
Apply Conjecture \ref{conj_conformal} to some global harmonic resolution $(\ol g, \ol \varphi)$, obtaining a metric $\ol g_0$ and a function $u$ 
that is harmonic with respect to $\ol g$, such that $\ol g_0 = u^4 \ol g$.  From
equation $(\ref{eq_conf_laplacian})$ in appendix B, the function $\ol \varphi_0:=\frac{\ol 
\varphi}{u}$ is harmonic with respect to $\ol g_0$.  
It is readily checked that $(\ol g_0, \ol \varphi_0)$ is a global harmonic resolution of $\Sigma$ obeying
the desired properties.
\end{proof}

\subsection{Proofs of the main theorems}
We prove the first part of Theorem $\ref{thm_zas_ineq}$.  Note that nonnegativity of the scalar curvature 
is preserved under global harmonic resolutions and applications
of Conjecture \ref{conj_conformal}, since the conformal factors are harmonic (see equation 
$(\ref{eq_conf_scalar_curv})$ in appendix B).
This fact will allow the use of the Riemannian Penrose inequality (Theorem \ref{rpi}).
\begin{proof} [Proof of first part of Theorem \ref{thm_zas_ineq}]
By hypothesis, we may assume the existence of a global harmonic resolution $(\ol g, \ol \varphi)$ 
of $\Sigma$ as in Conjecture \ref{conformal_conj_cor}.  As above, we let
$\tilde \Sigma$ be the outermost minimal area enclosure of $\Sigma$ in the metric $\ol g$.
(We reiterate the points that the Riemannian Penrose inequality applies to $(M, \ol g)$
and $|\Sigma|_{\ol g} = |\tilde {\Sigma}|_{\ol g}$.)  Now we make a series of estimates, 
where $m$ and $\ol m$ are the ADM masses of $g$ and $\ol g$:
{\allowdisplaybreaks
\begin{align*}
m &= \ol m - \frac{1}{2\pi} \lim_{r \to \infty}\int_{S_r} \ol\nu (\ol \varphi) \ol{dA}\\
		 &\qquad\qquad \text{(formula $(\ref{eq_conf_masses})$ in appendix 
B; $S_r$ is a large coordinate sphere)}\nonumber\\
	&\geq \sqrt{\frac{|\tilde{ \Sigma}|_{\ol g}}{16\pi}} - \frac{1}{2\pi} \int_{\Sigma} \ol\nu (\ol \varphi) \ol{dA}\\
		&\qquad \qquad \text{(Riemannian Penrose inequality; $\ol \Delta\ol \varphi=0$; Stokes' theorem)}\nonumber\\
	&\geq \sqrt{\frac{|\Sigma|_{\ol g}}{16\pi}} - \frac{1}{2\pi} \left(\int_{\Sigma} (\ol\nu (\ol \varphi))^{4/3} \ol{dA}\right)^{3/4} 
		|\Sigma|_{\ol g}^{1/4}\\
		&\qquad\qquad\text{(Conjecture \ref{conformal_conj_cor}, H\"{o}lder's inequality)}\nonumber\\
	&= \frac{1}{4}\sqrt{\frac{|\Sigma|_{\ol g}}{\pi}} - \frac{1}{2} \left(\frac{1}{\pi}\int_{\Sigma} (\ol\nu (\ol \varphi))^{4/3} 
	   \ol{dA}\right)^{3/4} \left(\frac{|\Sigma|_{\ol g}}{\pi}\right)^{1/4}\\
		&\qquad\qquad\text{(rearranging)}\nonumber\\
	&\geq \inf_{x \in \R} \left(\frac{1}{4}x^2 - \frac{1}{2} \left(\frac{1}{\pi}\int_{\Sigma} (\ol\nu (\ol \varphi))^{4/3}
           \ol{dA}\right)^{3/4}x\right)\\
	        &\qquad\qquad\text{(replacing $\left(\frac{|\Sigma|_{\ol g}}{\pi}\right)^{1/4}$ with $x$)}\\
	&= -\frac{1}{4} \left(\frac{1}{\pi}\int_{\Sigma} (\ol\nu (\ol \varphi))^{4/3}
           \ol{dA}\right)^{3/2}\\
	&\qquad\qquad\text{(the quadratic $\frac{1}{4}x^2+bx$ has minimum value $-b^2$)}\\
	&= m_{\text{ZAS}}(\Sigma)\\
	&\qquad\qquad\text{(definition of regular mass; Proposition \ref{prop_masses_agree})}\nonumber
\end{align*}}
This proves inequality $(\ref{zas_ineq})$.
\end{proof}
Before proving the second part of Theorem \ref{thm_zas_ineq}, we show that Theorem \ref{thm_zas_ineq_full} easily follows from the above.
\begin{proof}[Proof of Theorem \ref{thm_zas_ineq_full}]
If the capacity of $\Sigma$ is positive, then $m_{\text{ZAS}}(\Sigma) = -\infty$ by Theorem 
\ref{thm_pos_cap}.  Therefore $m \geq m_{\text{ZAS}}(\Sigma)$ follows trivially.
Now assume $\Sigma$ has zero capacity.

Let $\{\Sigma_n\}$ converge to $\Sigma$ in $C^1$, and let $\varphi_n$ be $g$-harmonic,
vanishing on $\Sigma_n$ and tending to 1 at infinity.
As before, $(g,\varphi_n)$ gives a global harmonic resolution of the ZAS $\Sigma_n$ of the 
metric $\varphi_n^4 g$.  By Theorem $\ref{thm_zas_ineq}$, we have
\begin{equation}
m_{ADM}(\varphi_n^4 g) \geq m_{\text{reg}}(\Sigma_n),
\label{mass_estimate}
\end{equation}
where $m_{ADM}(\varphi_n^4 g)$ is the ADM mass of the metric $\varphi_n^4 g$.  
Formula $(\ref{eq_conf_masses})$ in appendix B allows us to compute the ADM mass of a conformal 
metric in terms
of the original metric.  In the case at hand it gives
\begin{equation}
m= m_{ADM}(\varphi_n^4 g) + \frac{1}{2\pi} \lim_{r \to \infty} \int_{S_r} 
\nu(\varphi_n) dA.
\label{eq_masses_gen_case}
\end{equation}
Since $\varphi_n$ is harmonic and $S_r$ is homologous to $\Sigma_n$, Stokes' Theorem shows that the second term on the right-hand 
side is equal to $\frac{1}{2\pi} \int_{\Sigma_n} \nu(\varphi_n) dA$, which equals
$\frac{1}{2\pi} C(\Sigma_n)$.  Since $\Sigma$ has zero capacity, $\lim_{n \to \infty} C(\Sigma_n)=0$.
Now, taking $\limsup_{n\to \infty}$ of $(\ref{mass_estimate})$ and applying $(\ref{eq_masses_gen_case})$, we have 
\begin{equation*}
m \geq \limsup_{n \to \infty} m_{\text{reg}}(\Sigma_n).
\end{equation*}
By taking the supremum of this expression over all such $\{\Sigma_n\}$, we have that $m \geq m_{\text{ZAS}}(\Sigma)$.
\end{proof}

Now, we prove the second part of Theorem \ref{thm_zas_ineq} (characterizing the case of equality).
\begin{proof} [Proof of second part of Theorem \ref{thm_zas_ineq}]
One may readily show that the Schwarz\-schild ZAS metric $g=\left(1+\frac{m}{2r}\right)^4 
\delta$ (with $m < 0$ and $\delta$ the flat metric) on $\R^3 \sm B_{| m|/2}(0)$
has both ADM mass and ZAS mass
equal to $m$.  (For the latter, apply the definition of regular mass to the resolution
$\left(\delta, 1+\frac{m}{2r}\right)$ of $\Sigma$.)

Assume all components of $\Sigma=\partial M$ are harmonically regular ZAS.  Let $(\ol 
g, \ol \varphi)$ be a global harmonic resolution of $\Sigma$
as in Conjecture \ref{conformal_conj_cor}; in particular $\tilde{\Sigma}$
is minimal and $|\Sigma|_{\ol g} = |\tilde { \Sigma}|_{\ol g}$, (where, as before,
$\tilde{\Sigma}$ is the outermost minimal area enclosure of $\Sigma$ in the metric $\ol g$).
Assume that the ADM mass $m$ of $(M,g)$ equals $m_{\text{ZAS}}(\Sigma)$, so that 
equality holds at each step in the proof
of inequality $(\ref{zas_ineq})$.  In particular, the following must hold:
\begin{enumerate}
  \item $\ol m = \sqrt{\frac{|\tilde{ \Sigma}|_{\ol g}}{16\pi}}$, so equality holds in the Riemannian Penrose inequality. Thus,
    $\ol g$ is isometric to the Schwarzschild metric of mass $\ol m > 0$ \mbox{(on $\R^3 \sm B_{\ol m/2}(0)$)} outside of $\tilde{\Sigma}$.  (This 
does not complete
    the proof, since we have not yet understood the region inside of $\tilde {\Sigma}$.)
  \item $\ol \nu(\ol \varphi)$ is constant on $\Sigma$ (by the case of equality of H\"{o}lder's inequality).
  \item The minimum of the quadratic $\frac{1}{4}x^2+bx$ is attained, so $x = -2b$.  Thus,
    \begin{equation}
      \left(\frac{|\Sigma|_{\ol g}}{\pi}\right)^{1/4} = \left(\frac{1}{\pi}\int_{\Sigma} (\ol\nu (\ol \varphi))^{4/3}
           \ol{dA}\right)^{3/4}.
      \label{eq_min_quad}
    \end{equation}
  Squaring $(\ref{eq_min_quad})$, dividing by $-4$, and applying the definition of regular mass gives
  \begin{equation}
    -\sqrt{\frac{|\Sigma|_{\ol g}}{16\pi}} = m_{\text{ZAS}}(\Sigma).
    \label{assumption3}
  \end{equation} 
\end{enumerate}

Out strategy is to show $\tilde { \Sigma}$  and $\Sigma$ have the same capacity (in the $\ol g$ metric); since the former encloses the latter by
definition,
Lemma $\ref{lemma_cap}$ would imply that these surfaces are equal.  First,
\begin{align*}
C(\Sigma)^2 &= \left(\int_\Sigma \ol \nu(\ol \varphi) d \ol A\right)^2
		&&\text{(expression $(\ref{eq_cap_boundary})$ for  capacity)}\\
	    &= \frac{1}{4} \left(\frac{1}{\pi} \int_\Sigma \ol \nu(\ol \varphi)^{4/3} d \ol A\right)^{3/2} 
		\cdot 4 \pi^{3/2} |\Sigma|_{\ol g}^{1/2}
			&&\text{($\ol \nu(\ol \varphi)$ is constant on $\Sigma$)}\\	
	    &= -m_{\text{ZAS}}(\Sigma)\cdot 4 \pi^{3/2} |\Sigma|_{\ol g}^{1/2}
		&&\text{($m_{\text{ZAS}}(\Sigma)=m_{\text{reg}}(\Sigma)$)}\\
	    &=\pi|\Sigma|_{\ol g}
		&&\text{(by equation $(\ref{assumption3})$)}
\end{align*}
To compute the capacity of $\tilde { \Sigma}$, let $\psi$ be the $\ol g$-harmonic function that equals $0$ on $\tilde 
{ \Sigma}$ and tends to 1 at infinity.  In fact, $\psi(r)$ is given by $\frac{1- \frac{\ol m}{2r}}{1+ \frac{\ol m}{2r}}$.
(Here, we have used the identification of $\ol g$ outside of $\tilde { \Sigma}$ with the Schwarzschild metric.)
\begin{align*}
C(\tilde { \Sigma})^2 &= \left( \int_{\tilde { \Sigma}} \ol \nu(\psi) d \ol A\right)^2
		&&\text{(expression $(\ref{eq_cap_boundary})$ for  capacity)}\\
	    &= \left(\frac{1}{4\ol m} |\tilde{\Sigma}|_{\ol g}\right)^2
		&&\left(\ol \nu(\psi)|_{\tilde {\Sigma}} = \frac{1}{\left(1+\frac{\ol m}{2r}\right)^2}\frac{\partial}{\partial r}
		 \psi(r)\Big|_{r=\ol m/2}= \frac{1}{4\ol m}\right)\\
	    &= \pi |\tilde { \Sigma}|_{\ol g} &&\text{(equality in Riemannian Penrose inequality)}\\
	    &\leq \pi |\Sigma|_{\ol g}
			&&\text{(by definition of $\tilde { \Sigma}$)}\\
	    &= C(\Sigma)^2
			&&\text{(above computation of $C(\Sigma)$)}
\end{align*}
Therefore, $C(\tilde {\Sigma}) \leq C(\Sigma)$.  But we know from Lemma $\ref{lemma_cap}$ that the reverse inequality holds as well
(since $\tilde { \Sigma}$ encloses $\Sigma$ by definition).  By
the second part of Lemma $\ref{lemma_cap}$, we have that $\tilde { \Sigma} = \Sigma$.  Therefore $\ol \varphi = \psi$, so $g$
is given by 
\begin{equation*}
g = \ol \varphi^4 \ol g = \psi^4 \left(1+\frac{\ol m}{2r}\right)^4 \delta = \left(1-\frac{\ol 
m}{2r}\right)^4 \delta,
\end{equation*}
which is the Schwarzschild ZAS metric of mass $-\ol m=m$.
\end{proof}

\subsection{Example: Schwarzschild space with a cylinder}
\label{sec_schwarz_cyl}
We now give a global example of a manifold of positive ADM mass that contains a zero
area singularity; the construction begins with a Schwarzschild space with a cylinder appended
to its horizon.  For fixed parameters $\ol m > 0, L > 0$,  
let $M$ be the manifold with boundary $S^2 \times [\ol m /2 - L, \infty)$.  Give $M$ a 
metric $\ol g$ according to:
\begin{equation*}
\ol g = 
\begin{cases}
dr^2 + 4 \ol m^2 d\sigma^2,& \text{on } S^2 \times [\ol m /2-L, \ol m/2)\\
\left(1+ \frac{\ol m}{2r}\right)^4 (dr^2 + r^2 d\sigma^2), & \text{on } S^2 \times [\ol m /2, \infty)
\end{cases}
\end{equation*}
where $r$ parametrizes $[\ol m/2-L, \infty)$ and $d\sigma^2$ is the round metric of radius 1 on $S^2$.
The first region is a round cylinder of length $L$ with spherical cross sections of area $16\pi \ol m^2$.
The second region, diffeomorphic to $\R^3 \sm B_{\ol m /2}(0)$,
has the Schwarzschild metric of mass 
$\ol m$.  By construction, the sphere
$S^2 \times \{\ol m /2\}$ is the apparent horizon, and this metric is $C^{1,1}$.
For a diagram, see figure \ref{fig_extended_schwarz}.  

Let $\ol \varphi$ be the unique 
$\ol g$-harmonic function that equals $0$ on 
$\Sigma = \partial M = S^2 \times \{m/2-L\}$ and tends to 1 at infinity.  Then $\Sigma$ is a 
harmonically regular ZAS of $g := \ol \varphi^4 \ol g$ with global harmonic resolution $(\ol g, \ol \varphi)$.  In 
this example, we 
shall compute the mass of the manifold $(M,g)$ as well as mass of the ZAS $\Sigma$ and 
find that  they respect inequality $(\ref{zas_ineq})$ and are equal if and only if $L=0$.  

\begin{figure}[ht]
\caption{Schwarzschild space with a cylinder}
\begin{center}
\includegraphics[scale=0.5]{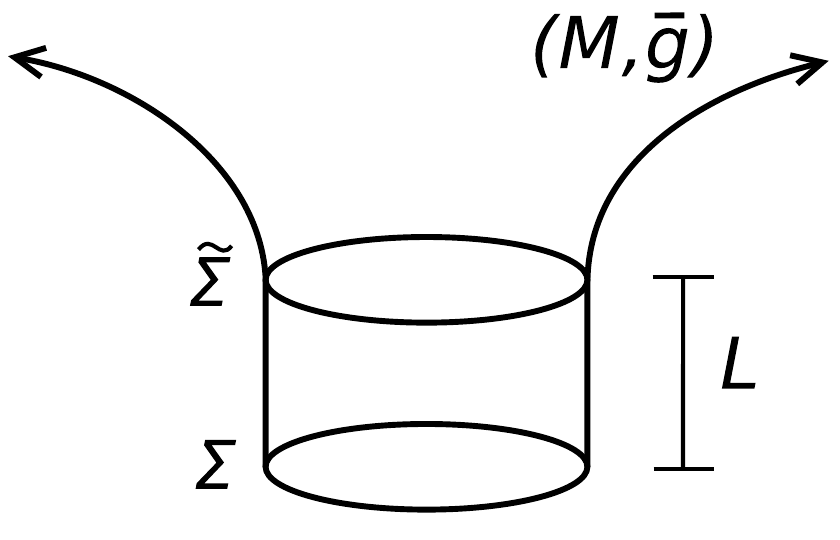}
\end{center}
\flushleft\footnotesize{$(M,\ol g)$ is Schwarzschild space with a round cylinder over $S^2$ of length $L$ appended to the minimal surface.}
\label{fig_extended_schwarz}
\end{figure}
By spherical symmetry, $\ol \varphi$ depends only on $r$.  Such harmonic functions on a round cylinder and Schwarzschild space
are readily characterized; $\ol \varphi$ takes the form:
\begin{equation*}
\ol \varphi(r) = 
\begin{cases}
\frac{a}{L}\left(r-\ol m / 2 +L\right), & r\in[\ol m /2-L, \ol m/2)\\
\frac{1+\frac{b}{2r}}{1+\frac{\ol m}{2r}}, & r \in [\ol m /2, \infty)
\end{cases}
\end{equation*}
where $a$ is a parameter to be determined and $b:=(2a - 1) \ol m$ is chosen so that $\ol \varphi$ is continuous (and so $g$ is continuous).
For $g$ to be a $C^1$ metric, the unit normal derivative $\ol \nu$ of $\ol \varphi$ at $\tilde \Sigma$ must be continuous.
Inside $\tilde \Sigma$, $\ol \nu(\ol \varphi) = \frac{a}{L}$.  Outside, 
$\ol \nu = \frac{1}{\left(1+\frac{\ol m}{2r}\right)^2} \frac{\partial}{\partial r}$, and a short computation
shows $\ol \nu(\ol \varphi) = \frac{1}{8\ol m}\left(1- \frac{b}{\ol m}\right)$ on $\tilde \Sigma$.  Equating these interior and 
exterior boundary conditions on $\ol \nu (\ol \varphi)$ allows us to solve for $a$:
\begin{equation*}
a = \frac{L}{L+4\ol m},
\end{equation*}
so that
\begin{equation*}
b:=(2a-1)\ol m = \left(\frac{2L}{L+4\ol m} - 1\right)\ol m=\frac{\ol m(L-4\ol m)}{L+4\ol m}.
\end{equation*}
Furthermore, $g = \ol \varphi^4 \ol g = \left(1+\frac{b}{2r}\right)^4 \delta$ outside $\tilde \Sigma$, so the ADM mass 
$m$ of $(M,g)$ is $b$.  
On the other hand, the mass of the harmonically regular ZAS $\Sigma$ can be computed from Definition \ref{def_reg_mass} and 
Proposition \ref{prop_masses_agree} as
\begin{align*}
m_{\text{ZAS}}(\Sigma) &= -\frac{1}{4}\left(\frac{1}{\pi} \int_{\Sigma} \ol \nu (\ol \varphi)^{4/3} d \ol A\right)^{3/2}\\
	  &= -\frac{|\Sigma|_{\ol g}^{3/2}}{4\pi^{3/2}} \left(\frac{a}{L}\right)^2\\
	  &= -\frac{|\tilde\Sigma|_{\ol g}^{3/2}}{4\pi^{3/2}} \frac{1}{(L+4\ol m )^2}\\
	  &= -\frac{16\ol m^3}{(L+4\ol m)^2},
\end{align*}
where we have used the fact that $|\Sigma|_{\ol g} = |\tilde \Sigma|_{\ol g} = 16\pi \ol m^2.$
Now, inequality $(\ref{zas_ineq})$ (the Riemannian ZAS inequality) is satisfied with $(M,g)$; by the
above computations, this is stated as:
\begin{equation*}
m =  \frac{\ol m(L-4\ol m)}{L+4\ol m}  \geq  -\frac{16\ol m^3}{(L+4\ol m)^2} = m_{\text{ZAS}}(\Sigma).
\end{equation*}
Note that equality holds if and only if the length $L$ of the cylinder is zero, in which case $(M,g)$
is the Schwarzschild ZAS.

Finally, we note that $g$ is merely $C^{1,1}$, so the ZAS theory of this paper does not technically apply.  However, $g$ could be perturbed 
to a $C^\infty$ metric exhibiting the qualitative properties above.

\begin{figure}[ht]
\caption{ADM mass and ZAS mass for Schwarzschild space with a cylinder}
\begin{center}
\includegraphics[scale=0.75]{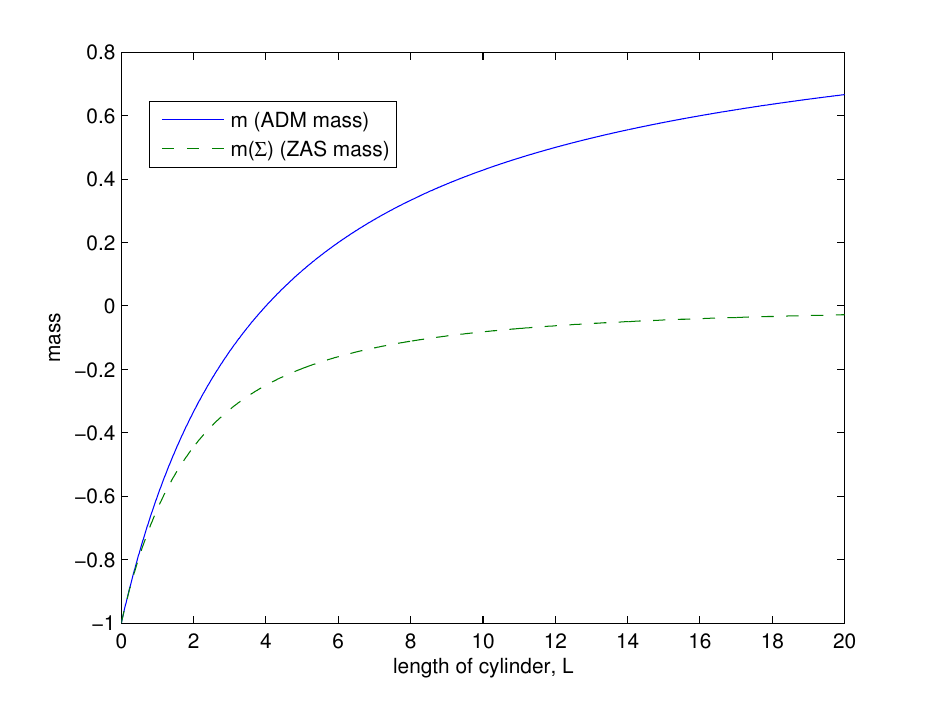}
\end{center}
\flushleft\footnotesize{With $\ol m=1$, the ADM mass $m$ of $(M,g)$ and the 
mass $m_{\text{ZAS}}(\Sigma)$ of the ZAS $\Sigma$ are plotted against the length $L$
of the cylinder, as described in section \ref{sec_schwarz_cyl}.  The former is asymptotic to 1, and the latter is 
asymptotic to zero.}  These values obey inequality (\ref{zas_ineq}) for all $L\geq 0$, with equality only for $L=0$.
\label{fig_zas_cylinder}
\end{figure}

\section{Conjectures and open problems}
\label{sec_conj}
Discussions in this paper raise a number of interesting questions.

\subsection{Conformal conjecture}
To complete the arguments of this paper, Conjecture \ref{conj_conformal} must be proved.  It is readily
verified in a spherically symmetric setting.
Establishing this conjecture as a theorem (or providing a counterexample) will be crucial to furthering 
the theory of ZAS.  For progress in this direction, see \cite{jauregui}.

\subsection{The case of equality in general}
\label{sec_coe_general}
For ZAS that are not harmonically regular, the case of equality of Theorem \ref{thm_zas_ineq_full} is not necessarily the Schwarzschild 
ZAS metric. For instance, $\R^3$ minus a finite number of points gives an 
asymptotically flat manifold of zero scalar curvature (and zero ADM mass) whose boundary consists of ZAS of zero mass. However we conjecture that
deleted points are the only obstruction to uniqueness of the Schwarzschild ZAS as the case of equality:
\begin{conj}
In Theorem \ref{thm_zas_ineq_full}, if it happens that $m = m_{\text{ZAS}}(\Sigma)$, then $(M,g)$ is isometric to 
\begin{compactenum}
\item a Schwarzschild metric with finitely many points deleted, if $-\infty < m < 0$, or
\item the flat metric on $\R^3$ with finitely many points deleted if $m=0$.
\end{compactenum}
\label{conj_coe}
\end{conj}
 
\subsection{Generalization of the Riemannian ZAS inequality}
The first author conjectured the following inequality for the situation in which black holes 
and ZAS exist simultaneously \cite{bray_npms}.  Recall that a surface is area outer-minimizing
if every surface enclosing it has equal or greater area.
\begin{conj}
Let $g$ be an asymptotically flat metric on $M$ with nonnegative scalar curvature.
Assume $M$ has compact smooth boundary $\partial M$ that is a
disjoint union $S \cup \Sigma$, where $S$ is an area outer-minimizing minimal surface and $\Sigma$
consists of ZAS.  Then the ADM mass $m$ of $(M,g)$ obeys
\begin{equation}
m \geq \sqrt{\frac{|S|_g}{16\pi}} + m_{\text{ZAS}}(\Sigma).
\label{eq_mixed_ineq}
\end{equation}
(Note that $S$ and $\Sigma$ need not be connected, and $m_{\text{ZAS}}(\Sigma)$ is nonpositive.)
\end{conj}
The heuristic behind this conjecture is simply that the Newtonian potential energy
between two bodies whose masses have opposite signs is positive, and thus ought to make a positive
contribution to the total (ADM) mass.  By this logic and the Riemannian Penrose and ZAS inequalities, the right hand side of $(\ref{eq_mixed_ineq})$
underestimates the ADM mass.
It is expected that equality would hold only in the cases of flat space, the Schwarzschild 
metric, and the Schwarzschild ZAS metric,
all with points possibly deleted (which correspond to removable ZAS of zero mass).  
See appendix B of \cite{robbins} for further discussion.

A special case to consider is when $\partial M$ consists of exactly two components -- one
connected ZAS and one connected area-outer-minimizing minimal surface -- whose masses are equal
in magnitude but opposite in sign.  The conjecture would predict that the ADM mass of $(M,g)$ is 
nonnegative.  Alternatively, such a manifold may be the most obvious source of a counterexample
to the conjecture.

\subsection{Alternate definitions of mass}
\label{sec_alt_mass}
J. Streets proved the existence and uniqueness of a (weakly defined) inverse mean curvature flow of surfaces ``out'' of
a zero area singularity $\Sigma$ (see \cite{streets}; full details available in unpublished version).  This gives rise to a canonical family of 
surfaces $\{\Sigma_t^*\}$ converging to 
$\Sigma$ as $t \to -\infty$.  Moreover, this family has the property that the limit of the Hawking mass is optimal in the
following sense:
\begin{equation}
\sup_{\{\Sigma_n\}} \limsup_{n \to \infty} m_H(\Sigma_n) = \lim_{t \to -\infty} m_H\left(\Sigma_t^*\right),
\label{zas_hawking_mass}
\end{equation}
where the supremum is taken over all sequences of surfaces $\{\Sigma_n\}$ converging in $C^2$ to $\Sigma$.
(This result assumes that the background metric has nonnegative scalar curvature.)
In other words, the existence of the canonical family obviates the need for taking the sup of a limsup.
Perhaps $(\ref{zas_hawking_mass})$ is a ``better'' definition of mass than that which
we have adopted.  However, it is unknown how to bound the ADM mass in terms of this quantity (unless $\Sigma$ consists of regular ZAS -- see
Corollary \ref{cor_reg_mass}).

We make the following related conjecture, which would show that the mass of ZAS defined via Hawking masses 
underestimates the definition of mass we have chosen.  It is motivated by the first
part of Proposition \ref{reg_mass_thm}.
\begin{conj}
For sequences $\{\Sigma_n\}$ converging in $C^1$ to a ZAS $\Sigma$ of a metric of nonnegative scalar curvature,
\begin{equation*}
\sup_{\{\Sigma_n\}} \limsup_{n \to \infty} m_H(\Sigma_n) \leq m_{\text{ZAS}}(\Sigma).
\end{equation*}
\end{conj}
The left hand side can be computed by equation
$(\ref{zas_hawking_mass})$ above.
In the case of $\Sigma$ connected, the
conjecture trivially holds for $\Sigma$ of positive capacity, since both sides equal $-\infty$ (Theorems \ref{thm_capacity} and \ref{thm_pos_cap}).  

Along another line of thought, recall the notion of removable $S^1$ singularities from section \ref{sec_removable}.  Our definition of mass
is unsatisfactory for such singularities, since it always $-\infty$.  A more sophisticated definition of mass would give
a finite quantity for removable $S^1$ singularities (and would ideally give a lower bound on the ADM mass in the same sense
as the Riemannian ZAS inequality.)

\subsection{Bartnik mass}
In this section we define the Bartnik mass of a collection of ZAS.
Let $(M,g)$ be an asymptotically flat manifold with boundary $\partial M$ such that $g$ has nonnegative
scalar curvature.  Let $S=\partial \Omega$ be a surface in $M$ that is a graph over $\partial 
M$ (see section \ref{sec_surfaces}).
Recall the \emph{outer Bartnik mass} of $S$ is defined by 
\cites{bray_RPI, bartnik3}
\begin{equation*}
m_{\text{outer}}(S) := \inf_{\mathcal{E}} \:m_{ADM}(\mathcal{E}),
\end{equation*}
where the infimum is taken over all ``valid extensions'' $\mathcal{E}$ of $S$ and $m_{ADM}$ is 
the ADM mass of $\mathcal{E}$.  An \emph{extension} of $S$
is an asymptotically flat manifold with boundary $(N,h)$ of nonnegative scalar curvature 
containing a subset that is
isometric to $\Omega$ and has no singularities outside of $\Omega$; it is said to be 
\emph{valid} if 
$S$ is not enclosed by a surface of less area.  

If $\Sigma = \partial M$ is a collection of ZAS, we define the (outer) \emph{Bartnik mass} of $\Sigma$ as
\begin{equation*}
m_{B}(\Sigma) := \lim_{n \to \infty} m_{\text{outer}}(\Sigma_n),
\end{equation*}
where $\Sigma_n$ is some sequence of surfaces converging in $C^0$ to $\Sigma$.
The limit exists (possibly $-\infty$) and is independent of the choice of sequence because of the monotonicity 
of the outer Bartnik mass (i.e., if $S_1$ encloses $S_2$, then $m_{outer}(S_1) \geq m_{outer}(S_2)$).  A priori,
the Bartnik mass of $\Sigma$ could be positive, negative, or zero.  In fact, there 
may exist examples of ZAS of positive Bartnik mass in a scalar-flat manifold with no 
apparent horizons.  However, the Riemannian ZAS inequality gives a lower bound on the Bartnik mass in 
terms of the mass of the ZAS (assuming Conjecture \ref{conj_conformal}).

Given a surface $S$, one version of the difficult and open ``Bartnik minimal mass extension problem''
is to determine whether there exists some valid extension $\mathcal{E}$
whose ADM mass equals the outer Bartnik mass of $S$.  However, the same problem for a collection of ZAS $\Sigma$ may in 
fact be more tractable. This is because any extension can be viewed as a valid extension, explained as follows.  
Let $\Sigma_n \to \Sigma$ in $C^0$; without loss of generality, we may assume that each surface enclosing $\Sigma_n$ has
area at least that of $\Sigma_n$, and the areas of $\Sigma_n$ decrease monotonically to zero.  Then, given
any extension $\mathcal{E}$ of $\Sigma_n$ (for some $n$), $\mathcal{E}$ gives a valid extension for some $\Sigma_{n'}$ with
$n' \geq n$.  Thus, for the purposes of computing the Bartnik mass of ZAS $\Sigma$, it is not necessary to restrict
only to valid extensions.

\subsection{Extensions to spacetimes and higher dimensions}
The definitions of ZAS, resolutions, mass, and capacity  extend naturally to higher 
dimensions (with suitable changes made to certain 
constants).  Also, the Riemannian Penrose inequality has been generalized to 
dimensions less than eight by the first author and D. Lee 
\cite{bray_lee}.  Thus, the Riemannian ZAS inequality in dimensions less than eight will follow readily if 
Conjecture $\ref{conj_conformal}$ is proved (in dimensions less than eight).

Ultimately, it would be desirable to go beyond the time-symmetric case and
develop a theory of ZAS for arbitrary spacelike slices of spacetimes.  
One might hope to prove a version of the ZAS inequality,
$m \geq m_{\text{ZAS}}(\Sigma)$, in this setting.  In figure \ref{fig_mass_theorems}, we illustrate how such an inequality would
fit in with the positive mass theorem and the Penrose inequality.  We refer the reader to a recent survey article
by the first author \cite{bray_survey}.

\begin{figure}[ht]
\caption{Theorems and conjectures in general relativity}
\begin{center}
\includegraphics[scale=0.5]{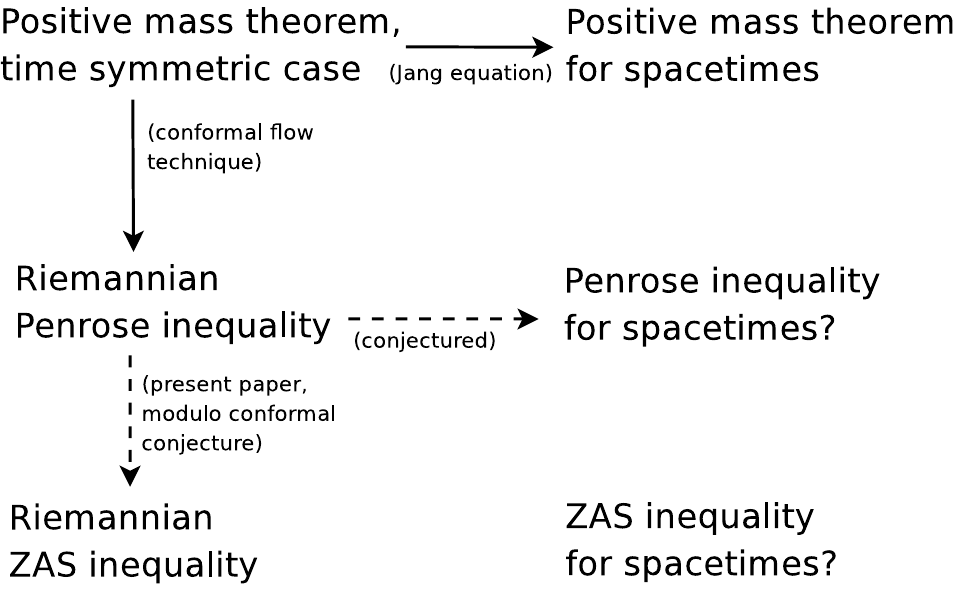}
\end{center}
\flushleft\footnotesize{This diagram illustrates the relationship between the positive mass theorem (upper left) and
some of its generalizations.  By solving the Jang equation, Schoen and Yau showed how to obtain a positive mass theorem for 
spacelike slices of spacetimes obeying the dominant energy condition \cite{schoen_yau_jang}.  In a different direction, the positive mass theorem was used by Bray to prove the Riemannian Penrose inequality (RPI) \cite{bray_RPI}.  (A version of the RPI proved by Huisken and Ilmanen \cite{imcf} did not use the positive mass theorem.) Next, the present paper used the RPI to prove the Riemannian ZAS inequality, modulo Conjecture 
\ref{conj_conformal}.  It is conjectured that the RPI may also be used to prove a version of the 
Penrose inequality for slices of spacetimes, using a generalization of the Jang equation \cites{braykhuri2, braykhuri}.  Lastly, a study of ZAS in 
spacetimes may lead to a type of ZAS inequality in this setting.}
\label{fig_mass_theorems}
\end{figure}

\appendix
\section{Asymptotically flat manifolds and ADM mass}
Global geometric problems in general relativity are often studied on asymptotically flat
manifolds.  Topologically, such spaces resemble $\R^3$ outside of a compact set, and
thus have well-defined notions of ``infinity.'' Geometrically, they approach Euclidean space 
at infinity.
\begin{definition}
A connected Riemannian 3-manifold $(M,g)$ (with or without boundary) is \textbf{asymptotically flat} 
(with one end) if there exists a compact subset $K \subset M$ and a diffeomorphism $\Phi: M 
\sm K \to \R^3 \sm \ol{B_1(0)}$ (where $\ol{B_1(0)}$ is the closed unit ball about the 
origin), 
such that in the coordinates $(x^1, x^2, x^3)$ on $M \sm K$ induced by $\Phi$, the metric 
satisfies, for some $p > \frac{1}{2}$ and $q > 3$,
\begin{equation*}
g_{ij} = \delta_{ij} + O(r^{-p}), \quad \partial_k g_{ij} = O(r^{-p-1}), 
\quad \partial_k \partial_l g_{ij} = O(r^{-p-2}),
\quad R_g = O(r^{-q})
\end{equation*}
for all $i,j,k,l \in \{1,2,3\}$,
where $r =|x|=\sqrt{(x^1)^2 + (x^2)^2 + (x^3)^2}$, $\partial_k = \frac{\partial}{\partial 
x^k}$, and $R_g$ is the scalar curvature of $g$.  The $(x^i)$ are called \textbf{asymptotically flat 
coordinates}.
\end{definition}
(This definition can easily be generalized to allow for several ends and higher dimensions.)  
Unless otherwise stated, it is assumed that $g$ is smooth and extends smoothly to the boundary.

For asymptotically flat manifolds, there is an associated quantity called the ADM 
mass that measures the rate at which the metric becomes flat
at infinity \cite{adm}.  Physically, it represents the mass of an isolated gravitational system.
\begin{definition}
The \textbf{ADM mass} of an asymptotically flat 3-manifold $(M,g)$ is the number
\begin{equation*}
m= \lim_{r \to \infty} \frac{1}{16\pi}\sum_{i,j=1}^3\int_{S_r} 
\left(\partial_i g_{ij} -
                \partial_j g_{ii}\right) \nu^j d\sigma,
\end{equation*}
where the $(x^i)$ are asymptotically  flat coordinates, $S_r$ is the coordinate sphere of radius $r=|x|$,
$\nu$ is the Euclidean unit normal to $S_r$ (pointing toward infinity), and $d\sigma$ is the Euclidean area form on $S_r$.
\label{def_adm_mass}
\end{definition}
Due to a theorem of Bartnik \cite{bartnik}, the mass is finite and depends only on $M$ 
and $g$ (and not the choice of coordinates).  It is left to the reader to show that the 
Schwarzschild metric $\left(1+\frac{m}{2r}\right)^4 \delta$ on $\R^3 \sm B_{|m|/2}(0)$ is 
asymptotically flat with ADM mass equal to $m$, for any real number $m$.  When there is
no ambiguity, we sometimes refer to the ADM mass as simply ``mass.''

\subsection{Outermost minimal surfaces}
\label{sec_omae}
A complete asymptotically flat manifold $(M,g)$ without boundary possesses a unique \emph{outermost 
minimal surface} $S$ (possibly empty, possibly with multiple connected components).  Each 
component of $S$ is minimal (i.e., has zero mean curvature) and is not enclosed by any other 
minimal surface. See figure \ref{fig_omae} for an illustration.
Some well-known (but not obvious) results are that 1) $S$ is a disjoint union
of smooth, immersed 2-spheres (that are in general embedded), 2) the region of $M$ outside of $S$ is 
diffeomorphic to $\R^3$ minus a finite number of closed balls, and 3) $S$ minimizes area in 
its homology class outside of $S$. See \cite{msy} and section 4 of \cite{imcf}
for further details and references.

Related to the above is the concept of outermost minimal area enclosure: let $(M,g)$ be 
an asymptotically flat manifold with smooth, compact boundary $\partial M$.  Then there exists a 
unique embedded surface $\tilde \Sigma$ (possibly of multiple connected components) enclosing $\partial M$ 
such that 1) $\tilde \Sigma$ has the least area among all surfaces that enclose $\partial M$ and 2) $\tilde \Sigma$ is not 
enclosed by a surface of equal area; 
such $\tilde \Sigma$ is called the \emph{outermost minimal area enclosure of $\partial M$} (see figure \ref{fig_omae}).  In general, 
all of the following cases may occur: i) $\tilde \Sigma= \partial M$, ii) $\tilde \Sigma \neq \partial M$ but $\tilde \Sigma \cap 
\partial M$ is nonempty, 
and iii) $\tilde \Sigma \cap \partial M$ is empty.  Note that $\tilde \Sigma$ may only have $C^{1,1}$ regularity
and may not have zero mean curvature (see Theorem 1.3 of \cite{imcf}).  However, $\tilde \Sigma \sm \partial M$ 
is $C^\infty$ with zero mean curvature (from standard minimal surface theory).

\begin{figure}[ht]
\caption{Outermost minimal surface and outermost minimal area enclosure}
\begin{center}
\includegraphics[scale=0.5]{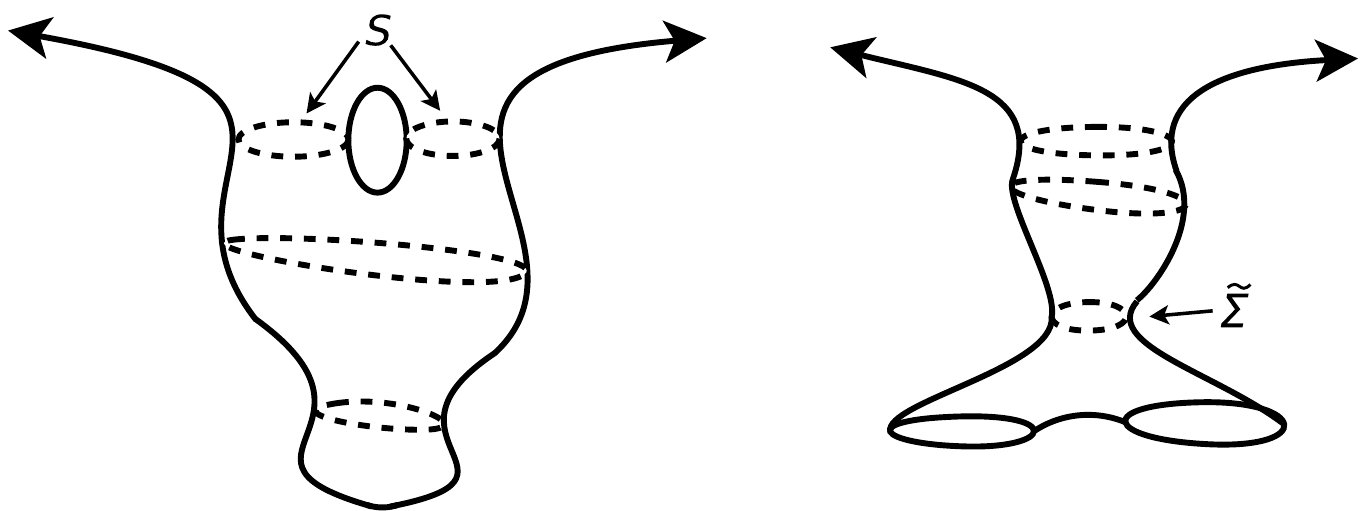}
\end{center}
\flushleft\footnotesize{Pictured on the left and right are asymptotically flat manifolds without and with boundary, respectively.
The dotted contours denote minimal surfaces.  
The surface labeled $S$, consisting of two components, is the outermost minimal surface.
The surface labeled $\tilde \Sigma$ is the outermost
minimal area enclosure of the boundary, and in this case it is disjoint from the boundary.}
\label{fig_omae}
\end{figure}

\section{Formulas for conformal metrics in dimension 3}
Here, we present several formulas that describe the behavior of certain
geometric quantities under conformal changes. Assume $g_1$ and $g_2$ are Riemannian 
metrics on a 3-manifold such that $g_2 = u^4 g_1$, where $u>0$.
\subsection{Laplacian}For any smooth function $\phi$, direct calculation shows
\begin{equation}
\Delta_{1} (u \phi) = u^5 \Delta_{2}(\phi) + \phi \Delta_{1} (u),
\label{eq_conf_laplacian}
\end{equation}
where $\Delta_{1}$ and $\Delta_{2}$ are the (negative spectrum) Laplace operators for $g_1$ and $g_2$, respectively 
\cite{bray_RPI}. In particular, if $\Delta_1 u = 0$, then $\Delta_2 (1/u)=0$.  More generally, 
if $\Delta_1 u = \Delta_1 \psi = 0$, then $\Delta_2 (\psi/u)=0$.
\subsection{Scalar curvature}
If $R_1$ and $R_2$ are the 
respective scalar curvatures of $g_1$ and $g_2$, then
\begin{equation}
R_2 = u^{-5}(-8\Delta_1 u + R_1 u).
\label{eq_conf_scalar_curv}
\end{equation}
In particular, if $u$ is harmonic with respect to $g_1$, then $g_1$ has nonnegative scalar
curvature if and only if $g_2$ does.
\subsection{Mean curvature}
If $S$ is a hypersurface of mean curvature $H_i$ with respect to $g_i$ (and outward unit normal $\nu_i$), $i=1,2$, then
\begin{equation}
H_2 = u^{-2} H_1 + 4u^{-3} \nu_1(u).
\label{eq_conf_mean_curv}
\end{equation}
Our convention is such that the mean curvature of a round sphere in flat $\R^3$ is positive.
\vspace{0.15in}
\subsection{ADM mass}
If $g_2=u^4 g_1$ are asymptotically flat with respective ADM masses $m_1$ and $m_2$,
and if $u \to 1$ at infinity, then
\begin{equation}
m_1 - m_2 = \frac{1}{2\pi} \lim_{r \to \infty} \int_{S_r} \nu(u) d\sigma,
\label{eq_conf_masses}
\end{equation}
where, in some asymptotically flat coordinate system $(x^i)$, $S_r$ is the coordinate sphere of radius $r=|x|$,
$\nu$ is the Euclidean unit normal to $S_r$ (pointing toward infinity), and $d\sigma$ is the Euclidean area form on $S_r$.


\begin{bibdiv}
\begin{biblist}

\bib{adm}{article}{
	title={Coordinate invariance and energy expressions in general relativity},
	author={R. Arnowitt},
	author={S. Deser},
	author={C. Misner},
	journal={Phys. Rev.},
	volume={122},
	date={1961},
	pages={997--1006}
}	

\bib{bartnik}{article}{
	title={The mass of an asymptotically flat manifold},
	author={R. Bartnik},
	journal={Comm. Pure Appl. Math.},
	volume={39},
	date={1986},
	pages={661--693}
}	

\bib{bartnik2}{article}{
	title={New definition as quasilocal mass},
	author={R. Bartnik},
	journal={Phys. Rev. Lett.},
	volume={62},
	date={1989},
	pages={2346--2348}
}	

\bib{bartnik3}{article}{
	title={Mass and 3-metrics of non-negative scalar curvature},
	author={R. Bartnik},
	journal={Proceedings of the ICM (Beijing, 2002)},
	volume={2},
	pages={231--240}
}	

\bib{bondi}{article}{
	title={Negative mass in general relativity},
	author={H. Bondi},
	journal={Rev. Modern Phys.},
	volume={29},
	number={3},
	date={1957},
	pages={423--428}
}	

\bib{bonnor}{article}{
	title={Negative mass in general relativity},
	author={W.B. Bonnor},
	journal={Gen. Relativity Gravitation},
	volume={21},
	number={11},
	date={1989},
	pages={1143--1157}
}	

\bib{bray_npms}{article}{
	title={Negative point mass singularities in general relativity},
	author={H.L. Bray},
	eprint={http://www.newton.ac.uk/webseminars/pg+ws/2005/gmr/0830/bray/},
	conference={
		title={Global problems in mathematical relativity},
		address={Isaac Newton Institute, University of Cambridge},
		date={2005-08-30}
	}
}	

\bib{bray_survey}{article}{
   author={H.L. Bray},
   title={On the positive mass, Penrose, and ZAS inequalities in general dimension},
   book={
      editor={H.L. Bray},
      editor={W.P. Minicozzi},
      series={Surveys in Geometric Analysis and Relativity},
      volume={20 ALM},
      publisher={Higher Education Press/International Press},
      place={Beijing/Boston},
   },
   date={2011}
}

\bib{bray_RPI}{article}{
	title={Proof the Riemannian Penrose inequality using the positive mass theorem},
	author={H.L. Bray},
	journal={J. Differential Geom.},
	volume={59},
	date={2001},
	pages={177--267}
}	

\bib{braykhuri2}{article}{
	title={A Jang equation approach to the Penrose inequality},
	author={H.L. Bray},
	author={M.A. Khuri},
        journal={Discrete Contin. Dyn. Syst.},
	date={2010},
        volume={27},
	number={2},
        pages={741--766}
}	


\bib{braykhuri}{article}{
	title={P.D.E.'s which imply the Penrose conjecture},
	author={H.L. Bray},
	author={M.A. Khuri},
        journal={Asian J. Math.},
	date={2011},
        volume={15},
        number={4},
        pages={557--610}
}	

\bib{bray_lee}{article}{
        title={On the Riemannian Penrose inequality in dimensions less than eight},
	author={H.L. Bray},
	author={D.A. Lee},
        journal={Duke Math. J.},
   volume={148},
   date={2009},
   number={1},
   pages={81--106},
}	

\bib{bray_miao}{article}{
	title={On the capacity of surfaces in manifolds with nonnegative scalar curvature},
	author={H.L. Bray},
	author={P. Miao},
	journal={Invent. Math.},
	volume={172},
	number={3},
	date={2008},
	pages={459--475}
}	

\bib{gibbons}{article}{
	title={On the stability of naked singularities with negative mass},
	author={G.W. Gibbons},
	author={S.A. Hartnoll},
	author={A. Ishibashi},
	journal={Progr. Theoret. Phys.},
	volume={113},
	number={5},
	date={2005},
	pages={963--978}
}	

\bib{gleiser}{article}{
	title={Instability of the negative mass Schwarzschild naked singularity},
	author={R.J. Gleiser},
	author={G. Dotti},
	journal={Classical Quantum Gravity},
	volume={23},
	date={2006},
	pages={5063--5077}
}	


\bib{imcf}{article}{
	title={The inverse mean curvature flow and the Riemannian Penrose inequality},
	author={G.  Huisken},
	author={T. Ilmanen},
	journal={J. Differential Geom.},
	volume={59},
	date={2001},
	pages={353--437}
}	


\bib{jauregui_hci}{article}{
	title={Invariants of the harmonic conformal class of an asymptotically flat manifold},
	author={J. Jauregui},
	journal={Comm. Anal. Geom.},
        volume={20},
        number={1},
	date={2012},
	pages={163--202}
}	

\bib{jauregui}{thesis}{
	author={J. Jauregui},
	title={Mass estimates, conformal techniques, and singularities in general relativity},
	type={Ph.D. thesis},
	organization={Duke University},
	date={2010}
}

\bib{jauregui_conf_flat}{article}{
	title={Penrose-type inequalities with a Euclidean background},
	author={J. Jauregui},
	eprint={http://arxiv.org/abs/1108.4042},
	date={2011}
}


\bib{lam}{thesis}{
	author={M.-K. G. Lam},
  title={The graph cases of the Riemannian positive mass and Penrose inequalities in all dimensions},
  type={Ph.D. thesis},
  organization={Duke University},
  date={2011}
}

\bib{msy}{article}{
	title={Embedded minimal surfaces, exotic spheres, and manifolds with positive Ricci curvature},
	author={W.H. Meeks III},
	author={L.M. Simon},
	author={S.-T. Yau},
	journal={Ann. of Math.},
	volume={116},
	date={1982},
	pages={621--659}
}	


\bib{penrose}{article}{
	title={Naked singularities},
	author={R. Penrose},
	journal={Ann. New York Acad. Sci.},
	volume={224},
	date={1973},
	pages={125--134}
}	

\bib{robbins}{thesis}{
	title={Negative point mass singularities in general relativity},
	author={N. Robbins},
	type={Ph.D. thesis},
	organization={Duke University},
	date={2007}
}

\bib{robbins_paper}{article}{
   author={N. Robbins},
   title={Zero area singularities in general relativity and inverse mean
   curvature flow},
   journal={Classical Quantum Gravity},
   volume={27},
   date={2010},
   number={2}
}
	

\bib{schoen_yau}{article}{
	title={On the proof of the positive mass conjecture in general relativity},
	author={R. Schoen},
	author={S.-T. Yau},
	journal={Comm. Math. Phys.},
	volume={65},
	date={1979},
	pages={45--76}
}	

\bib{schoen_yau_jang}{article}{
	title={Positivity of the total mass of a general space-time},
	author={R. Schoen},
	author={S.-T. Yau},
	journal={Phys. Rev. Lett.},
	volume={43},
	date={1979},
	pages={1457--1459}
}	

\bib{streets}{article}{
	title={Quasi-local mass functionals and generalized inverse mean curvature flow},
	author={J. Streets},
	journal={Comm. Anal. Geom.},
	volume={16},
	date={2008},
	pages={495--537}
}

\bib{witten}{article}{
	title={A new proof of the positive energy theorem},
	author={E. Witten},
	journal={Comm. Math. Phys.},
	volume={80},
	date={1981},
	pages={381-402}
}	

\end{biblist}
\end{bibdiv}
\end{document}